\documentclass{article}
\usepackage{graphicx} 
\usepackage{mathrsfs}
\usepackage{xcolor}
\usepackage{soul}
\usepackage{amsmath}
\usepackage{amsthm}
\usepackage{amssymb}
\usepackage{mathtools}
\usepackage{tikz}
\usetikzlibrary{cd,arrows.meta, decorations.pathmorphing, calc}
\newtheorem{theorem}{Theorem}[section]
\newtheorem{lemma}[theorem]{Lemma}
\newtheorem{proposition}[theorem]{Proposition}
\newtheorem{corollary}[theorem]{Corollary}

\newtheorem{assumption}[theorem]{Assumption}
\newtheorem*{theorem*}{Theorem}
\newtheorem{remark}[theorem]{Remark}
\newtheorem{example}[theorem]{Example}
\usepackage[hidelinks]{hyperref}
\usepackage[a4paper, total={6in, 8in}]{geometry}
\usepackage{stmaryrd}

\newcommand{\SigO}{\sigma(-\Delta_D(\Omega_-))}  

\newcommand{\intt}{\mathrm{int}}
\newcommand{\bH}{\mathbb{H}}
\newcommand{\bA}{\mathbb{A}}
\newcommand{\R}{\mathbb{R}}
\newcommand{\N}{\mathbb{N}}
\newcommand{\C}{\mathbb{C}}
\newcommand{\Hol}{H^{1,\mathrm{loc}}}
\newcommand{\Lloc}{L^{2}_{\mathrm{loc}}}
\newcommand{\Lcomp}{L^{2}_{\mathrm{comp}}}

\newcommand{\ri}{\mathrm{i}}
\newcommand{\cS}{\mathcal{S}}
\newcommand{\cA}{\mathcal{A}}
\newcommand{\sA}{\mathscr{A}}
\newcommand{\cH}{\mathcal{H}}

\newcommand{\cM}{\mathcal{M}}
\newcommand{\cB}{\mathcal{B}}

\newcommand{\supp}{\mathrm{supp}}

\newcommand{\rd}{\mathrm{d}}
\newcommand{\re}{\mathrm{e}}

\newcommand{\rea}{\mathrm{Re}\,}

\newcommand{\Trc}{\mathrm{Tr}}
\newcommand{\Trcr}{\mathrm{tr}}
\newcommand{\dimH}{\dim_H}

\newcommand{\Acal}{\mathcal{A}}
\newcommand{\Bcal}{\mathcal{B}}

\newcommand{\dx}{\mathrm dx}
\newcommand{\dy}{\mathrm dy}

\newcommand{\hookdoubleheadrightarrow}{\hookrightarrow\mathrel{\mspace{-15mu}}\rightarrow}
\newcommand{\capp}{\mathrm{cap}}

\renewcommand{\thefootnote}{\fnsymbol{footnote}}

\title{Integral Equation Methods for Scattering\\ by Multifractal Obstacles
}
\author{S. N. Chandler-Wilde\footnotemark[1], G. Claret\footnotemark[2], D. P. Hewett\footnotemark[3],\\ A. Rozanova-Pierrat\footnotemark[2],  and S. Sadeghi\footnotemark[1]}

\begin{document}

\maketitle

\footnotetext[1]{Department of Mathematics and Statistics, University of Reading, Whiteknights PO Box 220, Reading RG6 6AX, UK Emails: \tt{s.n.chandler-wilde@reading.ac.uk}, \tt{s.sadeghi@pgr.reading.ac.uk}}
\footnotetext[2]{CentraleSup\'elec, Universit\'e Paris-Saclay, 9 Rue Joliot Curie, 91190 Gif-sur-Yvette, France Emails: \tt{gabriel.claret@centralesupelec.fr}, \tt{anna.rozanova-pierrat@centralesupelec.fr}}
\footnotetext[3]{Department of Mathematics, University College London, 25 Gordon Street, London WC1H 0AY, UK Email: \tt{d.hewett@ucl.ac.uk}}

\renewcommand{\thefootnote}{\arabic{footnote}}
\setcounter{footnote}{0}

\newcommand{\sing}[1]{\int_{\Gamma}G_k(x-y)#1(y)\,d\sigma(y)}

\newcommand{\doub}[1]{\int_{\Gamma}\frac{\partial G_k(x-y)}{\partial \nu(y)}#1(y)\,d\sigma(y)}

\newcommand{\locsob}[1]{H^{1,\text{loc}}(#1)}

\newcommand{\tr}[1]{\text{tr}_\Gamma(#1)}

\newcommand{\Tr}[1]{\text{Tr}_{#1}}

\newcommand{\trad}[1]{\text{tr}^*_\Gamma(#1)}

\vspace{-5ex}

\begin{quote}
\small   Caetano et al.\ ({\em Proc. R. Soc. A.} {\bf 481}:20230650, 2025) have proposed a formulation for sound-soft acoustic scattering by a compact scatterer $O\subset \R^n$, in which the scattered field is represented as an acoustic Newtonian potential whose density is the solution of an operator equation on a compact set $\Gamma \subset O$. In the case that $\Gamma$ is Ahlfors-David $d$-regular (a $d$-set), for some $d\in (n-2,n]$, they show, moreover, that the operator equation can be interpreted as an integral equation, the integration with respect to $d$-dimensional Hausdorff measure, and present a convergent Galerkin scheme for numerical computation. 
In this paper we make a substantial extension of these results  so that they apply to more realistic fractal scatterers that are multifractal, in the sense that they have spatially varying fractal dimension. Firstly, we provide, inspired by  Claret et al.\ ({\em J.~Math.~Pures Appl.}\ {\bf 212}:103888, 2026), an interpretation of this operator equation as an equation between a trace space on $\Gamma$ and its dual, and, in many cases, relate the density to a notion of the normal derivative of the scattered field on $\Gamma$. Secondly, we show that the operator equation is equivalent to an integral equation on $\Gamma$ 
whenever $\Gamma$ is the support of a Radon measure $\mu$ such that: (i) the trace operator from $H^1(\R^n)$ to  $L^2(\Gamma,\mu)$ is continuous and; (ii) certain canonical singular integrals with respect to $\mu$ are finite; and we characterise a large class of measures for which (i) and (ii) hold. Finally, we show that Galerkin methods based on 
finite element subspaces of $L^2(\Gamma,\mu)$ are convergent if and only if, additionally, $C_0^\infty(\R^n\setminus \Gamma)$  is dense in the kernel of the trace operator. 
These results apply, in particular, if $\Gamma$ is a finite union of $d$-sets with different values of $d$.  In the case that each $d$-set is the attractor of an iterated function system of contracting similarities, we establish rates of convergence for the Galerkin method. 
\end{quote}

\tableofcontents

\section{Introduction} \label{sec:intro}

In this paper we study the solution, by integral equation methods, of the exterior sound-soft scattering problem for the Helmholtz equation in  the case where the scatterer $O$ is an arbitrary compact set, for example a union of compact sets each having a different fractal dimension (see Figure \ref{Fig:Setting}). In more detail, given some wavenumber $k>0$ and some compact obstacle $O\subset \R^n$, where $n\geq 2$ and
$$
\Omega:= O^c := \R^n\setminus O
$$ 
is connected, and given some incident field $u^{\mathrm{inc}}\in \Hol(\R^n)$ that satisfies $\Delta u^{\mathrm{inc}}+k^2u^{\mathrm{inc}}=0$
in some open neighbourhood
 of $O$, the {\em exterior sound-soft or Dirichlet scattering problem} is the problem of determining the scattered  field $u\in \Hol(\Omega)$ that satisfies the Helmholtz equation
\begin{equation} \label{eq:he}
\Delta u + k^2 u = 0 \quad \mbox{in} \quad \Omega,
\end{equation}
the Sommerfeld radiation condition (equation \eqref{eq:src} below), expressing that the scattered field is outgoing,
 and  the boundary condition  
 $$
 u^{\mathrm{tot}}:= u^{\mathrm{inc}}+u=0 \quad \mbox{on} \quad  \partial \Omega,
 $$ 
 in the sense that $u^{\mathrm{tot}}\in \Hol_0(\Omega)$. We will refer to $u^{\mathrm{tot}}$ as the {\em total field}, and direct the reader to \S\ref{sec:fs} for our function space notations.  
It is well known that this scattering problem is well--posed for all wavenumbers $k>0$ (one proof is via a standard weak formulation in $\Omega_R:= \{x\in \Omega:|x|<R\}$, e.g., \cite[Eqn.~(1.4)]{CWMonk2008}).

\begin{figure}[ht]
\centering
\begin{tikzpicture}[scale=0.8]
\draw[rotate=90, transform shape] (3,-5) node{\includegraphics[scale=0.7]{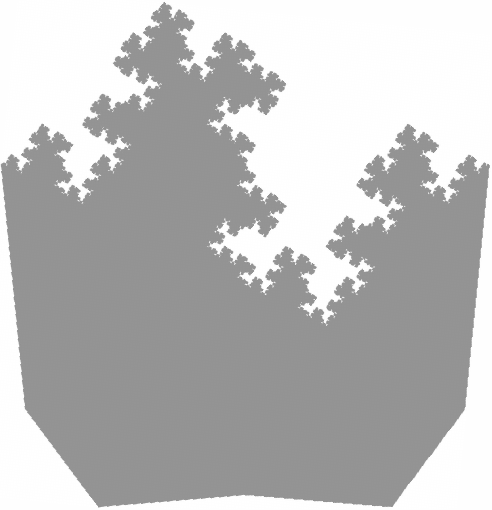}};
\draw[rotate=-15, transform shape] (2,0.7) node{\includegraphics[scale=0.5]{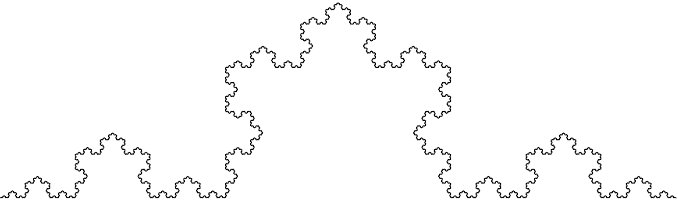}};
\draw (2.5,5.5) node {\Large $\Omega$} ;
\draw (7, 2) node {\large $O_1$} ;
\draw (-1.1,-1) node {\large $O_2$} ;
\draw[->, >=stealth] (-0.7,-1) parabola (0.35,-0.2) ;
\draw[->, blue!70!green, line width = 2pt] (-2,5) -- ++(2.5,-1.5) node[midway, circle, fill=white]{\large $u^{\mathrm{inc}}$} ;
\draw[->, blue!70!green, line width = 2pt] (-2+3*0.2,5+5*0.2) -- ++(2.5,-1.5) ;
\draw[->, blue!70!green, line width = 2pt] (-2-3*0.2,5-5*0.2) -- ++(2.5,-1.5) ;
\end{tikzpicture}
\caption{Illustration of the setting in a case where the obstacle $O$ is the union of two disjoint compact sets, $O_1$ and $O_2$. In the case shown $O_1$ has non-empty interior while the interior of $O_2$ is empty. The incident field propagates in $\Omega:= O^c$, the complement of $O$. The incident field $u^{\mathrm{inc}}$ is scattered by $O$, giving rise to the scattered field $u$. The total field $u^{\mathrm{tot}}=u^{\mathrm{inc}}+u$ vanishes on $\partial \Omega=\partial O=O_2\cup \partial O_1$.}
\label{Fig:Setting}
\end{figure}
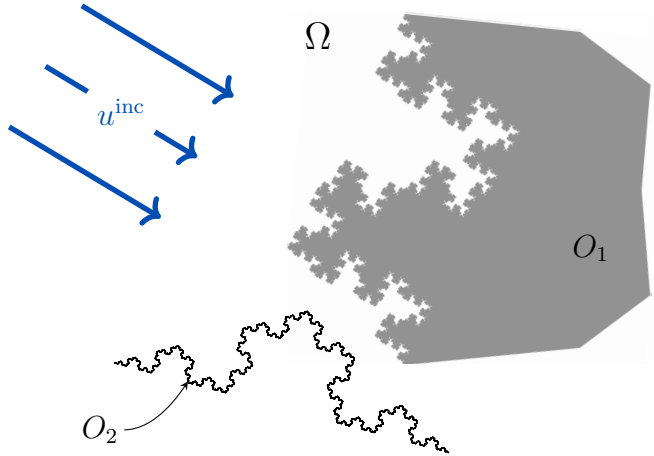

Recently, Caetano et al.~\cite{caetano2023integral} have proposed a reformulation of this scattering problem as an operator equation on some subset of $O$, that applies for a general compact obstacle $O\subset \R^n$. Let $\Phi(\cdot,\cdot)$ denote the (outgoing) fundamental solution of the Helmholtz equation \eqref{eq:he}, given by \eqref{eq:Phidef} below, so that 
\begin{equation} \label{eq:Phidef3}
\Phi(x,y) = \frac{\re^{\ri k|x-y|}}{4\pi|x-y|},  \qquad x,y\in \R^3, \;\; x\neq y,
\end{equation}
in the physically important case $n=3$.
 Let $\cA\psi$ denote the acoustic Newtonian potential with density $\psi$, defined by
\begin{equation} \label{eq:Newt}
\mathcal{A}\psi(x)=\int_{\mathbb{R}^n}\Phi(x,y)\psi(y)\,\rd y, \qquad  x\in \R^n.
\end{equation}

To obtain the formulation of \cite{caetano2023integral} we first choose a compact set $\Gamma$ with
\begin{equation} \label{eq:GamRest}
\partial O \subset \Gamma \subset O.
\end{equation}
(This geometry, including the constraint \eqref{eq:GamRest}, is illustrated in Figure \ref{Fig:Gamma}.)
 Having selected $\Gamma$
we seek a solution to the scattering problem in the form $u|_\Omega$, where $u=\cA\phi$ for some $\phi\in H^{-1}_\Gamma:= \{\psi\in H^{-1}(\R^n):\supp(\psi)\subset \Gamma\}$. As we recall in  \S\ref{sec:IE}, this ansatz satisfies the scattering problem if $\phi$ satisfies an operator equation
\begin{equation} \label{eq:iemain}
A\phi = g,
\end{equation}
where $A$ is a continuous mapping from $H^{-1}_\Gamma$ to a closed subspace  of $H^1(\R^n)$ that is a natural realisation of the dual space $(H_\Gamma^{-1})^*$, and $g$ is, roughly speaking, the projection of $-u^{\mathrm{inc}}$ onto that subspace. 

We emphasise that each choice of a compact $\Gamma$ satisfying \eqref{eq:GamRest} leads to a distinct
formulation \eqref{eq:iemain}. As we recall from \cite{caetano2023integral,SiavashSimon2} in \S\ref{sec:IE}, for every $\Gamma$ satisfying \eqref{eq:GamRest} and all $k>0$ equation \eqref{eq:iemain} has a solution and $A$ is Fredholm of index zero. However, for a particular $k>0$, uniqueness of solution of \eqref{eq:iemain} depends on the choice of $\Gamma$ as we discuss in Remark \ref{rem:choice} below. In particular, if $\Gamma=O$, uniqueness holds for all $k>0$, so that \eqref{eq:iemain} has exactly one solution for all $k>0$.   

\begin{figure}[ht]
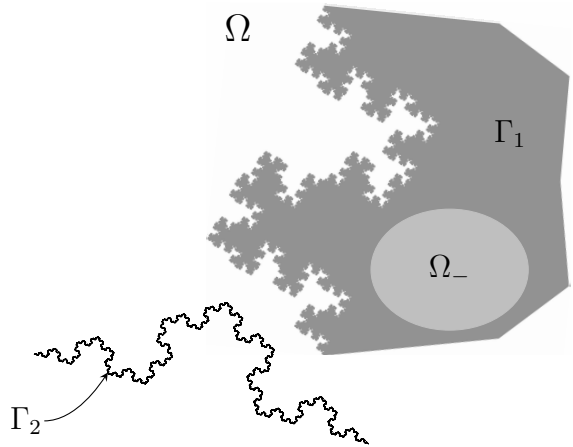

\centering
\begin{tikzpicture}[scale=0.8]
\draw[rotate=90, transform shape] (3,-5) node{\includegraphics[scale=0.7]{Mink6_CroppedBox_FilledNoLine.pdf}};
\draw[rotate=-15, transform shape] (2,0.7) node{\includegraphics[scale=0.5]{VK_10_Black.pdf}};
\filldraw[lightgray] (6, 1.5) ellipse (1.3 and 1) ;
\draw (6, 1.5) node {\large $\Omega_-$} ;
\draw (2.5,5.5) node {\Large $\Omega$} ;
\draw (7, 3.75) node {\large $\Gamma_1$} ;
\draw (-1,-1) node {\large $\Gamma_2$} ;
\draw[->, >=stealth] (-0.7,-1) parabola (0.35,-0.2) ;
\end{tikzpicture}
\caption{Illustration of one choice of a compact set $\Gamma$ satisfying~\eqref{eq:GamRest} in the case that $O = O_1\cup O_2$ is as in Figure~\ref{Fig:Setting}. The compact set $\Gamma$ is the disjoint union of two compacts sets, $\Gamma_2=O_2$ and $\Gamma_1$ which satisfies $\partial O_1\subset\Gamma_1\subset O_1$ so that $\Omega_-:=O\setminus\Gamma=O_1\setminus\Gamma_1$ and $O_1 = \Gamma_1\cup \Omega_-$. The compact set $\Gamma_1$ is shaded in dark grey and the open set $\Omega_-$ in light grey.}
\label{Fig:Gamma}
\end{figure}

From a theoretical perspective, these new  formulations provide an alternative route, starting from Theorem \ref{thm:invert} below, to proving existence of solution to the scattering problem, generalising proofs of existence of solution via integral equation methods that are well established in the case that the scatterer has a Lipschitz boundary (e.g., \cite[Thm.~9.11]{mclean2000strongly}). These new formulations are also very promising as a starting point for  computations, in particular in cases where the scatterer is fractal or has fractal boundary.
A key result of  \cite{caetano2023integral} is that, in the special case that $\Gamma$ is uniformly $d$-dimensional, for some $d\in (n-2,n]$ (precisely, $\Gamma$ is a $d$-set as defined
 in \S\ref{sec:meas} below), the equation \eqref{eq:iemain} can be reformulated as an integral equation, with integration with respect to  $d$-dimensional Hausdorff measure, $\cH^d$. In particular, in the case that $\Omega$ is a Lipschitz domain and 
$\Gamma = \partial O=\partial \Omega$, 
 \eqref{eq:iemain} is equivalent to a standard first kind boundary integral equation, with integration with respect to surface measure (which coincides with the measure $\cH^{n-1}$); see \cite[Rem.~3.17]{caetano2023integral} or Remark \ref{rem:kn1} below. This interpretation of \eqref{eq:iemain} as an integral equation leads to the possibility of numerical solution by Galerkin methods, where the Galerkin matrix is constructed by integration with respect to the measure $\cH^d$. Indeed, in the $d$-set case, convergence of such a Galerkin method, based on a piecewise-constant approximation, was proved in  \cite{caetano2023integral}. Moreover, numerical results were reported and rates of convergence established in the case when $\Gamma$ is the attractor of an iterated function system (IFS) of contracting similarities,  as defined in \S\ref{sec:meas}, satisfying the standard open set condition (OSC).

The purpose of the current paper is to investigate the extent to which the results of \cite{caetano2023integral}, described
in the above paragraph, extend to an arbitrary compact obstacle $O\subset \R^n$ and to an
arbitrary compact $\Gamma$ satisfying \eqref{eq:GamRest}. Our motivation is to understand and compute scattering
 by fractal obstacles 
 that have components of different fractal dimension or a fractal dimension that is otherwise spatially varying.
Clearly, for this programme to be possible, we need a measure on $\Gamma$; our standing assumption in \S\ref{sec:ie} and \S\ref{sec:GM} will be that $\Gamma$ is the support of some (positive) Radon measure $\mu$. Main results we establish are sufficient conditions on $\Gamma$ and $\mu$ which:
\begin{enumerate}
\item[(i)] \label{item:i} allow $u=\cA\phi$ to be written as an integral on $\Gamma$ with respect to the measure $\mu$ (Lemma \ref{lem:Aref}); 
\item[(ii)]  \label{item:ii} allow the operator equation \eqref{eq:iemain} to be written as an integral equation (where integration is with respect to the measure $\mu$) (Theorem \ref{prop:interep}); 
\item[(iii)] \label{item:iii} enable the coefficients in the Galerkin matrix to be written as integrals with respect to $\mu$ (Proposition \ref{prop:coeff}); 
\item[(iv)] ensure that the Galerkin method is convergent (Corollary \ref{cor:GM}).
\end{enumerate}  
\noindent  To a large extent these main results also tease out necessary conditions on $\Gamma$ and $\mu$ for the above properties to hold. 
In particular, we will show in \S\ref{sec:trL2} that a sufficient condition for all of (i)--(iii) to hold is that the measure $\mu$ is upper $d$-regular on $\Gamma$, in the sense of \eqref{eq:udr}, for some $d\in (n-2,n]$. 

The above results depend on a detailed understanding of 
the trace operator $\widetilde \Trc_\Gamma:H^1(\R^n)\to L^2(\Gamma,\mu)$ and its adjoint $\widetilde \Trc_\Gamma^*:L^2(\Gamma,\mu)\to H^{-1}(\R^n)$, defined in \S\ref{sec:trL2} below. In our Galerkin method, generalising the scheme in \cite{caetano2023integral}, our approximate solution takes the form $\widetilde \Trc_\Gamma^* \widetilde \phi_N$, where $\widetilde \phi_N\in L^2(\Gamma,\mu)$ is piecewise constant; this construction assumes that the operator $\widetilde \Trc_\Gamma^*:L^2(\Gamma,\mu)\to H^{-1}(\R^n)$ is well-defined and continuous, equivalently that the same holds for $\widetilde \Trc_\Gamma:H^1(\R^n)\to L^2(\Gamma,\mu)$, which is our Assumption \ref{ass:cont}. 

This assumption is already enough to guarantee (i), i.e., Lemma \ref{lem:Aref}. Further, assuming Assumption \ref{ass:cont}, and that $h_N\to 0$ as $N\to\infty$, where $h_N$ is the 
maximum element diameter for the piecewise-constant approximation $\widetilde \phi_N$, we show in Corollary \ref{cor:GM} that the Galerkin method is convergent 
 if and only if $C_0^\infty(\Gamma^c)$ is dense in  $\ker \widetilde \Trc_\Gamma$, the kernel of $\widetilde \Trc_\Gamma$, where $\Gamma^c:=\R^n\setminus \Gamma$. 
A further main result of the paper, building on results announced in \cite[\S5]{Hinz}, is to characterise this property of $\ker \widetilde \Trc_\Gamma$ (see Proposition \ref{prop:equiv}), in particular showing its equivalence  to Assumption \ref{ass:Gabriel} below, a subtle property of the measure on $\Gamma$, that a function vanishes $\mu$-a.e.~on $\Gamma$ if and only if it vanishes quasi-everywhere (in the usual sense of \S\ref{sec:meas}). Sufficient conditions for this property to hold have been obtained recently by Hinz et al.\ \cite{Hinz}.

Applying the results from \cite{Hinz}, we show that all of (i)-(iv) hold if $\Gamma$ is a finite union of $d$-sets (with different values of $d$), and $\mu$ is correspondingly a sum of Hausdorff measures of different dimensions, supported on the corresponding subsets of $\Gamma$ (see Theorem \ref{thm:summ}). In the final parts of \S\ref{sec:GM} we consider the special case that, additionally, these different subsets of $\Gamma$ are disjoint, in which case we prove a regularity result for the solution of \eqref{eq:iemain} and rates of convergence for the Galerkin method.

Let us briefly summarise the rest of the paper. In \S\ref{sec:IE} we recall the results from Caetano et al.\ \cite{caetano2024integral} that we need, as refined recently in \cite{SiavashSimon2}. In \S\ref{sec:trace}, as a new contribution, we reformulate our main operator equation \eqref{eq:iemain} as an equation in trace spaces on $\Gamma$ (Theorem \ref{thm:trace}). This is an important starting point for our main results summarised above. We also, in \S\ref{sec:normal}, show that the solution to this reformulated operator equation can be interpreted as a normal derivative of the total field, defined in a weak sense, in the case that $\Omega$ is an extension domain (which includes many cases where $\partial \Omega=\partial O$ is fractal). 
In \S\ref{sec:ie} we study the trace operator $\widetilde \Trc_\Gamma^*$ and prove the results at (i) and (ii) above and the characterisations of $\ker \widetilde \Trc_\Gamma$ in Proposition \ref{prop:equiv}. In \S\ref{sec:GM1} we introduce our Galerkin method and prove the results at (iii) and (iv) above. In \S\ref{sec:onset} we analyse a Galerkin method implementation for the case where $\partial O$ is a $d$-set, for some $d<n$, and $\intt(O)$, the interior of $O$, is an open $n$-set as defined in \S\ref{sec:onset}. Because it has guaranteed convergence for all $k>0$ and asymptotically the same computational cost as standard Galerkin boundary element methods, this scheme may be of interest for numerical computation even in the standard case where $\Omega$ is Lipschitz so that $d=n-1$ (see Remark \ref{remkn3}). We establish our results on regularity of the solution in \S\ref{sec:reg} and use these to prove convergence rates in \S\ref{sec:conv}.
So as not to disrupt the flow of the main text, we collect certain preliminary materials in the appendix; we recall certain basic results regarding measures, $d$-sets, and iterated function systems in \S\ref{sec:meas}, and definitions and properties of function spaces in \S\ref{sec:fs}.

We discuss the implementation of our Galerkin method, using the quadrature rules of \cite{Gibbsetal23,Gibbsetal24,Jolyetal24} and extensions of these quadrature rules, only briefly, in Remark \ref{rem:ni}. Details and numerical experiments are left to a future paper. We refer the reader to \cite{RP25} for a wider review, that builds in part on the results of this paper, of measure-free and boundary measure-based trace operator approaches  to the formulation of elliptic PDEs on domains with fractal boundaries. (These are approaches using the trace operators $\Trc$ and $\widetilde \Trc$, respectively, introduced in \S\ref{sec:trace} and \S\ref{sec:trL2} below.)   

{\bf Notations.} Throughout, $\N$ and $\N_0$ denote the sets of positive and non-negative integers, respectively. For $G\subset \R^n$, $G^c:= \R^n\setminus G$ denotes the complement of $G$, $|G|$ the ($n$-dimensional) Lebesgue measure of $G$, and $\dimH(G)$ the Hausdorff dimension of $G$ (see \S\ref{sec:meas}). For $r>0$ and $x\in \R^n$, $B_r(x)$ denotes the closed ball $B_r(x):= \{y\in \R^n:|x-y|\leq r\}$.

\section{The operator equation} \label{sec:IE}

In this section we recall in more detail the formulation \eqref{eq:iemain} of \cite[\S3(a)]{caetano2023integral} and its recent refinements in \cite{SiavashSimon2}, using the function space notations that we recall in \S\ref{sec:fs}. As we have already noted, the solution to \eqref{eq:iemain}, while supported on $\Gamma$, is a distribution in $H^{-1}(\R^n)$. In \S\ref{sec:trace} we make a new formulation of this operator equation in terms of a trace space of functions whose domain is $\Gamma$, and the dual of this trace space. We also interpret,  in \S\ref{sec:normal}, the solution of this new formulation, in the case when $\Omega$ is an extension domain,  as a generalised normal derivative of the solution to the scattering problem.

As we recalled in \S\ref{sec:intro}, given some wavenumber $k>0$, some compact obstacle $O\subset \R^n$  such that $\Omega=O^c$ is connected, and some incident field $u^{\mathrm{inc}}\in \Hol(\R^n)$ that satisfies \eqref{eq:he} in a neighbourhood of $O$, the approach of Caetano et al.\ \cite{caetano2023integral} to the exterior sound-soft scattering problem in $\Omega$ is to choose a compact $\Gamma\subset \R^n$ satisfying \eqref{eq:GamRest} and look for a solution in the form $u|_\Omega$, where $u=\cA \phi$ for some $\phi\in H^{-1}_\Gamma$. Here $\cA\phi$ is the acoustic Newtonian potential, defined by  \eqref{eq:Newt} for $\phi\in L^1_{\mathrm{comp}}(\R^n)\supset \Lcomp(\R^n)$,  where $\Phi(\cdot,\cdot)$ is the fundamental solution of the Helmholtz equation, given by (e.g., \cite[Eqn.~(9.14)]{mclean2000strongly})
\begin{equation} \label{eq:Phidef}
\Phi(x,y) := \frac{\ri}{4}\left(\frac{k}{2\pi |x-y|}\right)^{(n-2)/2}H_{(n-2)/2}^{(1)}(k|x-y|), \qquad x,y\in \R^n, \;\; x\neq y,
\end{equation} 
where $H_m^{(1)}$ denotes the Hankel function of the first kind of order $m$. We note that, for $\phi\in L^1_{\mathrm{comp}}(\R^n)$, the integral \eqref{eq:Newt} exists for a.e.\ $x\in \R^n$ by Young's inequality, e.g., \cite[Thm.~3.1]{mclean2000strongly}, indeed, the integral exists for all $x\not\in \supp(\phi)$, and that, by \cite[Eqn.~10.16.2]{NIST}, \eqref{eq:Phidef} reduces to the familiar \eqref{eq:Phidef3} in the case $n=3$.

 It is standard (e.g., \cite[Lem.~3.24]{mclean2000strongly}) that $\Lcomp(\R^n)$ is dense in $H^{s}_{\mathrm{comp}}(\R^n)$ for every $s<0$ and (e.g., \cite[Thm.~3.1.2]{SaSc:11}, \cite[Thm.~6.1]{mclean2000strongly}) that, for every $s\in \R$, $\cA$ extends to a continuous map  
\begin{equation} \label{eq:mapping}
\cA:H^{s}_{\mathrm{comp}}(\R^n) \to H^{s+2,\mathrm{loc}}(\R^n).
\end{equation} 
Further, for every $\phi\in H^{-1}_\Gamma$, since $\Gamma \subset O$, $\cA\phi\in C^\infty(\Gamma^c)$ satisfies \eqref{eq:he} (e.g., \cite[Eqns.~(6.2), (9.14)]{mclean2000strongly}); indeed, for every $s\in \R$,
\begin{equation} \label{eq:inv}
(\Delta  + k^2)\cA\phi = \cA(\Delta  + k^2)\phi = -\phi, \qquad \phi\in H^s_{\mathrm{comp}}(\R^n).
\end{equation}
Moreover, $\cA\phi$ satisfies the Sommerfeld radiation condition,
\begin{equation} \label{eq:src}
\frac{\partial u(x)}{\partial r} - \ri k u(x) = o(r^{-(n-1)/2}) \quad \mbox{as} \quad r:= |x|\to \infty,
\end{equation}
uniformly in $x/r$ (e.g.,  \cite[Lem.~7.14, Thm.~9.6]{mclean2000strongly}).

Explicitly, for $\phi\in H_\Gamma^{-1}$ and $\chi\in C^\infty_{0,\Gamma}$,
\begin{equation} \label{eq:Aphi2}
\cA\phi(x) = \langle \phi, \overline{\chi \Phi(x,\cdot)}\rangle, \qquad x\not\in \supp(\chi),
\end{equation}
 where $\langle\cdot,\cdot\rangle$ is the duality pairing on $H^{-1}(\R^n)\times H^1(\R^n)$ given by \eqref{eq:dual} and, for every compact $F\subset \R^n$,
\begin{equation} \label{eq:CG}
C^\infty_{0,F} := \{\chi\in C_0^\infty(\R^n): \chi=1 \mbox{ in a neighbourhood of }F\}.
\end{equation}
For certainly \eqref{eq:Aphi2} holds by \eqref{eq:Newt} if $\phi\in \Lcomp(\R^n)$ and $\supp(\phi)\subset \supp(\chi)$.  If $\phi\in H_\Gamma^{-1}$ then there exists a sequence $(\phi_j)\subset \Lcomp(\R^n)$ with $\phi_j\to \phi$ in $H^{-1}(\R^n)$ and $\supp(\phi_j)\subset \supp(\chi)$, for each $j$, so that, by continuity, \eqref{eq:Aphi2} holds also for $\phi\in H_\Gamma^{-1}$.

It follows from these observations that, if $u:=\cA \phi$ with $\phi\in H^{-1}_\Gamma$, then $u\in \Hol(\R^n)$, so that $u^{\mathrm{tot}}:= u+u^{\mathrm{inc}}\in \Hol(\R^n)$, and that $u|_\Omega\in \Hol(\Omega)$ satisfies the exterior sound-soft scattering problem if and only if $u^{\mathrm{tot}}|_\Omega\in \Hol_0(\Omega)$.  Further, $u^{\mathrm{tot}}|_\Omega\in \Hol_0(\Omega)$ if  $u^{\mathrm{tot}}|_{\Gamma^c}\in \Hol_0(\Gamma^c)$, since $\Omega\subset \Gamma^c$ and $\partial \Omega = \partial O\subset \Gamma$.  
Thus (see \cite[\S4]{SiavashSimon2} for details)
$u^{\mathrm{tot}}|_{\Gamma^c}\in \Hol_0(\Gamma^c)$ if and only if $(\chi u^{\mathrm{tot}})|_{\Gamma^c}\in H_0^1(\Gamma^c)$, for some  $\chi\in C^\infty_{0,\Gamma}$.
Thus, fixing  some $\chi\in C^\infty_{0,\Gamma}$, it holds that $u^{\mathrm{tot}}|_\Omega\in \Hol_0(\Omega)$  if $\chi u^{\mathrm{tot}}\in \widetilde H^1(\Gamma^c)$, which is equivalent to $P(\chi u^{\mathrm{tot}}) = 0$, where
\begin{equation} \label{eq:Porth}
P:H^1(\R^n)\to V_1(\Gamma^c) := (\widetilde H^{1}(\Gamma^c))^\perp = (H^{-1}_\Gamma)^*
\end{equation}
is orthogonal projection. The identification of $V_1(\Gamma^c)$ with the dual space $(H^{-1}_\Gamma)^*$ is via the duality pairing $\langle\cdot,\cdot\rangle$ on $H^{1}(\R^n)\times H^{-1}(\R^n)$, restricted to $V_1(\Gamma^c)\times H^{-1}_\Gamma$; see \S\ref{sec:subspaces}.

We have shown all but the last sentence of the following result (cf.~\cite[Thm.~3.4, Rem.~3.8]{caetano2023integral}); the last sentence is shown as \cite[Prop.~4.2]{SiavashSimon2}. Note that the definition \eqref{eq:Akdef} of $A$ in this result  is independent of the choice of $\chi$, since $(\chi_1-\chi_2)\Psi \in \widetilde H^1(\Gamma^c)$, so that $P((\chi_1-\chi_2)\Psi)=0$, for all $\Psi\in \Hol(\R^n)$ and $\chi_1,\chi_2 \in C^\infty_{0,\Gamma}$. Thus the following proposition and all subsequent results hold for every  choice of $\chi \in C^\infty_{0,\Gamma}$. We remark that it is possible to choose $\chi\in C^\infty_{0,\Gamma}$ so that $u^{\mathrm{inc}}$ satisfies \eqref{eq:he} in a neighbourhood of $\mathrm{supp}(\chi)$, so that $\chi u^{\mathrm{inc}}\in C^\infty_0(\R^n)$.

\begin{proposition} \label{prop:ie1} 
Suppose that the compact set $\Gamma\subset \R^n$ satisfies \eqref{eq:GamRest} and that $\phi\in H^{-1}_\Gamma$. Then $u|_\Omega$, where $u:=\cA \phi$, satisfies the exterior sound-soft scattering problem if 
\begin{equation} \label{eq:iemain2}
A\phi = g := -P(\chi u^{\mathrm{inc}}),
\end{equation}
where $A:H^{-1}_{\Gamma}\to V_1(\Gamma^c)=(H^{-1}_\Gamma)^*$ is defined by
\begin{equation} \label{eq:Akdef}
A\psi:=P(\chi \mathcal{A}\psi), \qquad \psi \in H^{-1}_\Gamma.
\end{equation}
Conversely, if $u\in \Hol(\Omega)$ satisfies the exterior sound-soft scattering problem and the definition of $u$ is extended to $\R^n$ by setting $u:=-u^{\mathrm{inc}}$ on $O=\R^n\setminus \Omega$, then $u\in \Hol(\R^n)$ and $u=\cA\phi$, where $\phi:= -(\Delta+k^2)u\in H^{-1}_{\partial O}\subset H^{-1}_\Gamma$ satisfies \eqref{eq:iemain2}.
\end{proposition}

Since the exterior sound-soft scattering problem has a unique solution for all $k>0$, the above proposition has the following corollary (\cite[Cor.~4.3]{SiavashSimon2} and cf.~\cite[Prop.~3.6]{caetano2023integral}). 

\begin{corollary} \label{cor:exists}
Suppose that the compact set $\Gamma\subset \R^n$ satisfies \eqref{eq:GamRest}. Then \eqref{eq:iemain2} has a solution $\phi\in H^{-1}_{\partial O} \subset H^{-1}_\Gamma$, for all $k>0$. 
\end{corollary}

Invertibility of $A:H^{-1}_\Gamma\to V_1(\Gamma^c)=(H^{-1}_\Gamma)^*$ is studied in \cite{caetano2023integral,SiavashSimon2} via a variational formulation of \eqref{eq:iemain2}. Using the duality pairing $\langle\cdot,\cdot\rangle$ defined in \eqref{eq:dual} in \S\ref{sec:fs}, define the sesquilinear form $a(\cdot,\cdot)$ associated to $A$ on  $H^{-1}_\Gamma \times H^{-1}_\Gamma$ by
$$
a(\varphi,\psi) := \langle A\varphi,\psi\rangle, \qquad \varphi,\psi\in H^{-1}_\Gamma.
$$
It is standard and immediate that, given $\phi\in H_\Gamma^{-1}$ and $g\in V_1(\Gamma^c)$, $\phi$ satisfies \eqref{eq:iemain2} if and only if $\phi$ satisfies the equivalent variational formulation
\begin{equation} \label{eq:iemainV}
a(\phi,\psi) = \langle g,\psi\rangle, \qquad \psi\in H_\Gamma^{-1}.
\end{equation}

The following result, showing that $A$ is a compact perturbation of a continuous, coercive operator, is proved as \cite[Lem.~3.3]{caetano2023integral} for the cases $n=2,3$, and extended to all $n\geq 2$ in \cite[Prop.~4.4]{SiavashSimon2}. (For detail of our terminology, of compactness and coercivity of operators and associated sesquilinear forms, see, e.g., 
\cite[\S2.2]{CWHeMoBe:21}.) 
\begin{proposition} \label{prop:coer}
Suppose that the compact set $\Gamma\subset \R^n$ satisfies \eqref{eq:GamRest}. 
Then, for each $k>0$, the sesquilinear form $a(\cdot, \cdot)$ is continuous and compactly perturbed coercive on  $H^{-1}_\Gamma \times H^{-1}_\Gamma$. Indeed, for some constants $C_{ \mathrm{cont}}, \alpha > 0$, and some compact sesquilinear form $\tilde a(\cdot, \cdot)$ on $H_\Gamma^{-1}\times H_\Gamma^{-1}$,
\begin{equation} \label{eq:form}
|a(\varphi,\psi)| \leq C_{\mathrm{cont}} \|\varphi\|_{H^{-1}(\R^n)}\, \|\psi\|_{H^{-1}(\R^n)}, \quad \rea(a(\psi,\psi)-\tilde a(\psi,\psi)) \geq \alpha \|\psi\|^2_{H^{-1}(\R^n)}, \qquad \varphi,\psi\in H^{-1}_\Gamma.
\end{equation}
\end{proposition}

Let
\begin{equation} \label{eq:Omegam}
\Omega_-:= O\setminus \Gamma,
\end{equation} 
and, in the case $\Omega_-\neq \emptyset$, let $\SigO$ denote the spectrum of $-\Delta_D(\Omega_-)$, the negative Dirichlet Laplacian on $L^2(\Omega_-)$, 
otherwise let $\SigO := \emptyset$. Recall (e.g., \cite{levitin2023topics}) that  $\SigO$ is a countable subset of $(0,\infty)$,  whose only accumulation point is $+\infty$, that consists only of eigenvalues, so that $k^2\in \SigO$ if and only if there exists a non-zero $v\in H_0^1(\Omega_-)$ that satisfies \eqref{eq:he} in $\Omega_-$. The following theorem, one part of \cite[Thm.~3.4]{caetano2023integral} (and see \cite[Thm.~4.5]{SiavashSimon2}), is, in part, a corollary of the above proposition.

\begin{theorem} \label{thm:invert}
Suppose that the compact set $\Gamma\subset \R^n$ satisfies \eqref{eq:GamRest}. Then $A:H^{-1}_\Gamma\to  V_1(\Gamma^c)=(H^{-1}_\Gamma)^*$ is Fredholm of index zero. Further, $A$ is injective, and so invertible, if and only if $k^2\not\in \SigO$. 
\end{theorem}

Of course this result implies that $A$ is invertible for all $k>0$ in the case that $\Gamma=O$. Combined with Corollary \ref{cor:exists} and Proposition \ref{prop:ie1}, Theorem \ref{thm:invert} has the following straightforward corollary.

\begin{corollary}\label{cor:IEfinal}(\cite[\S4.6]{SiavashSimon2}) 
Suppose that the compact set $\Gamma\subset \R^n$ satisfies \eqref{eq:GamRest}. Then \eqref{eq:iemain2} has a solution $\phi\in H^{-1}_{\partial O} \subset H^{-1}_\Gamma$. Indeed, this is the unique solution $\phi\in H^{-1}_\Gamma$ of \eqref{eq:iemain2} if $k^2\not\in \SigO$, while \eqref{eq:iemain2} has infinitely many solutions  $\phi\in H^{-1}_\Gamma$ if $k^2\in \SigO$. If  $\phi\in H^{-1}_\Gamma$ satisfies \eqref{eq:iemain2} and $u=\cA\phi$, then $u|_\Omega\in \Hol(\Omega)$ is the unique solution of the exterior sound-soft scattering problem.
\end{corollary}

\begin{remark}[When is the solution just $u=0$?] \label{rem:zero}
It is possible that the only solution $\phi\in H_\Gamma^{-1}$ to \eqref{eq:iemain2} is $\phi=0$, so that, by Corollary \ref{cor:IEfinal}, the unique solution of the exterior sound-soft scattering problem is $u=0$. In particular this is the case if $H_\Gamma^{-1}=\{0\}$, which holds if and only if $\mathrm{cap}(\Gamma)=0$  (see \S\ref{sec:subspaces}), where $\mathrm{cap}(\Gamma)$ is the capacity of $\Gamma$, defined in \eqref{eq:capdef}.  Further,  given \eqref{eq:GamRest},  so that $\Gamma \supset \partial O$, $\mathrm{cap}(\Gamma)=0$ if and only if $\mathrm{cap}(O)=0$. For if $\intt(O)=\emptyset$, then $\Gamma=O$, while if $\intt(O)\neq \emptyset$, then $\mathrm{cap}(O)>0$ and also $\dim_H(\partial O) \geq n-1$, by \eqref{eq:boundH}, so that  $\mathrm{cap}(\partial O)>0$, by \eqref{eq:cap}, so that $\mathrm{cap}(\Gamma)>0$.
\end{remark}

Let us comment on the choice of $\Gamma$; see also Remark \ref{remkn3} below and \cite[Rem.~3.8]{caetano2023integral}.
\begin{remark}[The choice of $\Gamma$] \label{rem:choice}
Each choice of $\Gamma$ satisfying \eqref{eq:GamRest} corresponds to a different operator equation formulation \eqref{eq:iemain2} of the problem.  If $\partial O \subset \Gamma_1\subset \Gamma_2 \subset O$, then, for each $k_0>0$ and $j=1,2$, the equation \eqref{eq:iemain2} corresponding to $\Gamma_j$ is, by Theorem \ref{thm:invert}, invertible except for a finite number $N_j(k_0)$  values of $k\in(0,k_0]$. Further, $N_2(k_0)\leq N_1(k_0)$ by standard monotonicity properties of Dirichlet eigenvalues (e.g., \cite[Thm.~3.2.1]{levitin2023topics}), in particular $N_2(k_0)=0$ if $\Gamma_2=O$. 
Thus, for each $k_0>0$, a larger choice of $\Gamma$ can lead to a failure of invertibility for fewer $k\in (0,k_0]$.
 On the other hand, if \eqref{eq:iemain} is used for numerical solution, as in \S\ref{sec:GM} below, a larger $\Gamma$ may require more degrees of freedom for a finite-element type discretisation. 

In the classical case that $O$ is the closure of a bounded Lipschitz open set $\Omega_-$, equation \eqref{eq:iemain2} with $\Gamma=\partial O=\partial \Omega_-$ is equivalent to a standard (and popular) boundary integral equation formulation of the scattering problem; see \cite[Rem.~3.17]{Caetano24} and Remarks \ref{rem:kn1} and \ref{rem:kn2} below. In this classical case the values of $k>0$ such that $k^2\in \SigO$ are often termed {\em irregular frequencies}. In the case that $\Gamma = \partial O \cup \Gamma'$, for some compact set  $\Gamma'\subset \Omega_-$, the formulation  \eqref{eq:iemain2} is reminiscent of the CHIEF method and its variants (e.g., \cite{Schenk68,Wu91}) for removing irregular frequencies of BIEs.
\end{remark}

\subsection{The operator equation reformulated in trace spaces} \label{sec:trace}

In this section we write down, in Theorem \ref{thm:trace}, a reformulation of the operator equation \eqref{eq:iemain} in terms of a trace space on $\Gamma$ and its dual. We first introduce the trace space $\cB(\Gamma)$ that we use and its dual space $\cB^*(\Gamma)$.

Suppose that $F\subset \R^n$ is closed. Following, e.g., \cite{claret2024layer,RP25}, let us denote  by $\cB(F)$ the vector space of all q.e.~equivalence classes of pointwise restrictions $\tilde u|_{F}$ of quasi-continuous representatives $\tilde u$ of classes $u\in H^1(\R^n)$ (these standard notions are recalled in \S\ref{sec:cap}). Following, e.g.,  \cite[Defn.~6.1.6]{adams2012function}, let us define a trace operator $\Trc_F:H^1(\R^n)\to \cB(F)$ by the rule that, for $u\in H^1(\R^n)$, $\Trc_Fu$ denotes the equivalence class in $\cB(F)$ that contains the pointwise restriction of every quasi-continuous representative $\tilde u$ of $u$. 
 Importantly, Adams and Hedberg \cite[Cor.~9.1.4]{adams2012function} characterise the kernel of $\Trc_F$, showing that
\begin{equation} \label{eq:ker}
\ker \Trc_F = \widetilde H^1(F^c),
\end{equation}
as defined in \eqref{eq:tildeHs}. It is convenient, in the case that $F$ is compact, to extend the domain of $\Trc_F$ from $H^1(\R^n)$ to $\Hol(\R^n)$, by setting
\begin{equation} \label{eq:extend}
\Trc_F u := \Trc_F(\chi_F u), \qquad u\in \Hol(\R^n),
\end{equation}
where $\chi_F\in C^\infty_{0,F}$ and $C^\infty_{0,F}$ is as defined in \eqref{eq:CG}. (It follows from \eqref{eq:ker} that this definition is independent of the particular $\chi_F$ that we choose.)

As discussed in \cite{claret2024layer}, equipped with the norm
\begin{equation} \label{eq:Bnorm}
\|f\|_{\cB(F)} := \min \{ \|v\|_{H^1(\R^n)} : v\in H^1(\R^n) \mbox{ and } \Trc_F v=f\},
\end{equation}
the space $\cB(F)$ is a Hilbert space, and it is clear that  $\Trc_F:H^1(\R^n)\to \cB(F)$ is bounded with $\|\Trc_F\| = 1$.
Generalising our notation in \eqref{eq:Porth}, let 
\begin{equation} \label{eq:V1def}
V_1(F^c) := (\widetilde H^1(F^c))^\perp, 
\end{equation}
so that
\begin{equation} \label{eq:perp}
H^1(\R^n) = \widetilde H^1(F^c) \oplus V_1(F^c).
\end{equation}
Then, as observed in  \cite{claret2024layer},
$$
V_1(F^c) = \{u\in H^1(\R^n):\Delta u = u \mbox{ in }F^c\}
$$
is the space of $1$-harmonic functions on $F^c$.
Further, where $\Trcr_F := \Trc_F|_{V_1(F^c)}$, i.e., $\Trcr_F$ is the restriction of $\Trc$ to $V_1(F^c)$, $\Trcr_F:V_1(F^c)\to \cB(F)$ is a unitary isomorphism. Define $E_F: \cB(F)\to H^1(\R^n)$ by $E_Ff := \Trcr_F^{-1}f$, $f\in \cB(F)$, so that $E_F$ is an isometry with range $V_1(F^c)$ and is a right inverse of $\Trc_F$. As a consequence of \eqref{eq:ker} we 
note that, for $f\in \cB(F)$, $u:= E_Ff$ is the unique solution of the Dirichlet boundary value problem:
\begin{equation}\label{Eq:Dir-1H-Prob}
\mbox{Find } u\in H^1(\R^n) \mbox{ such that } \Delta u = u \mbox{ in } F^c \mbox{ and } \Trc_Fu = f.
\end{equation}
Let $\cB^*(F)$ denote the dual space of $\cB(F)$ and $\Trcr_F^*:\cB^*(F)\to H^{-1}_F = (V_1(F^c))^*$ the adjoint of $\Trcr_F$, so that
\begin{equation} \label{eq:adjointy}
\langle \phi,\Trcr_F u \rangle_{\cB^*(F)\times \cB(F)} = \langle  \Trcr_F^*\phi,u \rangle, \qquad  \phi\in \cB^*(F), \;u\in V_1(F^c),
\end{equation}
and note that $\Trcr_F^*$ is also a unitary isomorphism. Similarly, let $\Trc_F^*:\cB^*(F)\to H^{-1}(\R^n)$ denote the adjoint of $\Trc_F$, and note that $\Trc_F^*\phi=\Trcr_F^*\phi$, $\phi\in \cB^*(F)$.

We can now deduce the promised reformulation of \eqref{eq:iemain2}, which is \eqref{eq:iemainB} in the following theorem. Clearly, \eqref{eq:iemainB} is a reformulation in which the unknown is a distribution on $\Gamma$, an element of the dual space $\cB^*(\Gamma)$. In contrast, the solution of \eqref{eq:iemain2}, while supported on $\Gamma$, is an element of $H^{-1}(\R^n)$, so a distribution on $\R^n$.  Note the representation $\sA = \Trc_\Gamma \cA \Trc^*_\Gamma$ for the operator $\sA$ in \eqref{eq:iemainB}, where $\Trc^*_\Gamma:\cB^*(\Gamma)\to H^{-1}(\R^n)$ is the adjoint of $\Trc$, which has range $H^{-1}_\Gamma$. This representation is analogous to a standard definition  (e.g., \cite[(7.3), p.~202]{mclean2000strongly}) of the single-layer boundary integral operator $S$ on the boundary of a Lipschitz domain.
\begin{theorem} \label{thm:trace}
Suppose that $\phi\in H_\Gamma^{-1}$  and $f:= (\Trcr^*_\Gamma)^{-1} \phi \in \cB^*(\Gamma)$. Then $\phi$ satisfies \eqref{eq:iemain2} if and only if $f$ satisfies
\begin{equation} \label{eq:iemainB}
\sA f = h:= \Trcr_\Gamma g = -\Trc_\Gamma u^{\mathrm{inc}},
\end{equation} 
where $\sA:\cB^*(\Gamma)\to \cB(\Gamma)$ is given by
\begin{equation} \label{eq:sAdef}
\sA := \Trcr_\Gamma A \Trcr^*_\Gamma = \Trc_\Gamma \cA \Trc^*_\Gamma.
\end{equation}
Further, $\sA$ is Fredholm of index zero, and is invertible if and only if $k^2\not\in \SigO$. For every $k>0$ \eqref{eq:iemainB} has a solution $f\in \cB^*({\partial \Omega})$ and, if $u:=\cA \Trcr^*_\Gamma f$, then $u|_\Omega$ is the unique solution of the exterior sound-soft scattering problem.
\end{theorem}
\begin{proof}
Suppose that $\phi\in H_\Gamma^{-1}$  and $f:= (\Trcr^*_\Gamma)^{-1} \phi \in \cB^*(\Gamma)$. Then, recalling that $\Trcr^*_\Gamma:\cB^*(\Gamma)\to H_\Gamma^{-1}$ and $\Trcr_\Gamma:V_1(\Gamma^c)\to \cB(\Gamma)$ are unitary isomorphisms, $\phi$ satisfies \eqref{eq:iemain2} if and only if $f$ satisfies $\sA f=h$. That $\sA$ is Fredholm of index zero, and is invertible if and only if $k^2\not\in \SigO$, follows by Theorem \ref{thm:invert}. The final sentence follows from Corollary \ref{cor:IEfinal}. The second equality in \eqref{eq:sAdef} follows since $\Trcr^*f=\Trc^*f$, for $f\in \cB^*(\Gamma)$, and from \eqref{eq:Akdef} and \eqref{eq:extend} which imply, for $\psi\in H_\Gamma^{-1}$,  that $\Trcr_\Gamma A \psi =  \Trc_\Gamma  P(\chi \cA \psi) = \Trc_\Gamma (\chi \cA \psi)= \Trc_\Gamma \cA \psi$, since $\Trc_\Gamma  (I-P)=0$ as the range of $I-P$ is  $\widetilde H^1(\Gamma^c)=\ker \Trc_\Gamma$. The same arguments give the last equality in \eqref{eq:iemainB}.
\end{proof}

\subsection{The solution as a normal derivative} \label{sec:normal}

The aim of this subsection is to interpret the solution $f\in\Bcal(\Gamma)$ to the operator equation \eqref{eq:iemainB}.
In this subsection only, we assume 
$\Omega$ is an extension domain in the sense of~\cite{hajlasz_sobolev_2008}, i.e.\ there exists a linear and continuous extension operator $\mathrm{E}:H^1(\Omega)\to H^1(\R^n)$, so that
\begin{equation*}
(\mathrm{E}u)|_\Omega=u\quad\mbox{and}\quad\|\mathrm Eu\|_{H^1(\R^n)}\le c\|u\|_{H^1(\Omega)}, \qquad  u\in H^1(\Omega),
\end{equation*}
with $c>0$ independent of $u$.
Note that this setting is fairly general: $\Omega$ could, for example,  be an $(\varepsilon,\delta)$-domain (see~\cite[Section 2]{rozanova-pierrat_generalization_2021} and the references therein), so that $\partial\Omega$ could be fractal, indeed of spatially varying (local) Hausdorff dimension.
In particular, the normal vector to $\partial \Omega$ need not be defined \textit{a priori}, so that the normal derivative need not be defined in the strong sense.

For $u\in H^1(\Omega)$, we define its trace on ${\partial \Omega}$ (still denoted by $\mathrm{Tr}_{\partial \Omega} u$) by $\mathrm{Tr}_{\partial \Omega} u:=\mathrm{Tr}_{\partial \Omega}(\mathrm Eu)$, noting that this definition is independent of the choice of the extension operator $\mathrm E$: indeed,
\begin{equation} \label{eq:Tracesame}
\mbox{if} \quad v\in H^1(\R^n) \quad \mbox{and}  \quad u=v|_\Omega, \quad \mbox{then} \quad \Trc_{\partial \Omega} u = \Trc_{\partial \Omega} v;
\end{equation}
see \cite[Rem.~6.2]{Biegart:09} or  \cite[Thm.~4.2, Lem.~4.3(i)]{Hinz}.
This gives rise to a linear, continuous and surjective operator $\mathrm{Tr}_{\partial \Omega}:H^1(\Omega)\to\Bcal({\partial \Omega})$.
For $u\in H^1(\Omega)$ with $\Delta u\in L^2(\Omega)$, we define the (weak) normal derivative $\partial_\nu u$ as the unique $\varphi\in\Bcal^*({\partial \Omega})$ such that
\begin{equation} \label{Eq:Norm-der}
 \langle\varphi,f\rangle_{\Bcal^*({\partial \Omega})\times\Bcal({\partial \Omega})}=\int_{\Omega}(\Delta u)\overline{v_f}\,\dx+\int_{\Omega}\nabla u\cdot\nabla \overline{v_f}\,\dx, \quad f\in \Bcal({\partial \Omega}),
\end{equation}
where $v_f:=E_{\partial \Omega} f|_\Omega$; recall that $E_{\partial \Omega}:\Bcal({\partial \Omega})\to H^1(\R^n)$ is defined above \eqref{Eq:Dir-1H-Prob}.
In the case of a smooth boundary, the sign convention in~\eqref{Eq:Norm-der} corresponds to considering a normal vector field $\nu$ on $\partial \Omega$ directed out of $\Omega$, and the weak normal derivative coincides with the usual normal derivative, which we denote by $\frac{\partial u}{\partial\nu}$.
Importantly, it follows from~\eqref{eq:ker} that a weak version of Green's first identity holds (see for instance~\cite[Remark 4.2]{hinz2023boundary}):
\begin{equation}\label{Eq:Green}
\langle\partial_\nu u,\mathrm{Tr}_{\partial \Omega} v\rangle_{\Bcal^*({\partial \Omega})\times\Bcal({\partial \Omega})}=\int_{\Omega}(\Delta u)\bar v\,\dx+\int_{\Omega}\nabla u\cdot\nabla \bar v\,\dx, \quad  v\in H^1(\Omega).
\end{equation}
It is convenient to extend the definition of $\partial_\nu u$ to  the case that $u\in \Hol(\Omega)$ with $\Delta u \in \Lloc(\Omega)$, defining $\partial_\nu u:= \partial_\nu (\chi u)$, where $\chi\in C_{0,O}^\infty$, as defined in \eqref{eq:CG}. (This definition  of $\partial_\nu u$ is independent of the choice of $\chi$.)

In this setting, we can give an interpretation of the solution $f\in\Bcal^*(\Gamma)$ to equation~\eqref{eq:iemainB} in terms of the total field $u^{\mathrm{tot}}=u+u^{\mathrm{inc}}$, where $u\in \Hol(\Omega)$ is the unique solution to the scattering problem. We assume in the first instance that $\Gamma=\partial \Omega$ and turn to the case of an arbitrary compact $\Gamma$ satisfying \eqref{eq:GamRest} in Theorem \ref{cor:normal} below. Recall we assume throughout that $u^{\mathrm{inc}}$ satisfies the Helmholtz equation in some open neighbourhood of $O$.
\begin{proposition} \label{prop:normal}
Assume $\Gamma=\partial\Omega$ where $\Omega$ is an extension domain, and that $k^2\not\in \sigma(-\Delta_D(\Omega_-))$.
Then the unique solution $f\in\Bcal^*(\Gamma)$ to the operator equation~\eqref{eq:iemainB} can be expressed in terms of the total field $u^{\mathrm{tot}}$ as
\begin{equation*}
f=\partial_\nu u^{\mathrm{tot}}.
\end{equation*}
\end{proposition}

\begin{proof}
Since the incident field $u^{\mathrm{inc}}$ solves the Helmholtz equation in some open neighbourhood of $O$, there exists a domain $D$ (which we can choose smooth) such that
\begin{equation*}
O\subset D \qquad\mbox{and}\qquad \Delta u^{\mathrm{inc}}+k^2u^{\mathrm{inc}}=0\mbox{ on }D.
\end{equation*}
Let $x\in D\cap\Omega$ and let $\chi\in C^\infty_{0,O}$ be such that $x\notin\operatorname{supp}\chi\subset D$.
Let $\varepsilon>0$ be small enough so that $B_\varepsilon(x)\cap\operatorname{supp}\chi=\emptyset$.
Applying Green's classical second identity (e.g.,~\cite[Eqn.~(3.5)]{ColKre}) to $(1-\chi)u^{\mathrm{tot}}\in C^\infty(\overline{D})$ and $\Phi(x,\cdot)\in C^\infty(\R^n\setminus\{x\})$ on $D\setminus B_\varepsilon(x)$ yields
\begin{multline*}
\int_{\partial D\cup\partial B_\varepsilon(x)}\left[\frac{\partial\Phi(x,y)}{\partial\nu(y)}u^{\mathrm{tot}}(y)-\Phi(x,y)\frac{\partial u^{\mathrm{tot}}(y)}{\partial\nu}\right]\mathrm d\sigma(y)\\
=\int_{D\setminus B_\varepsilon(x)}\Phi(x,y)\big(2\nabla\chi(y)\cdot\nabla u^{\mathrm{tot}}(y)+\Delta\chi(y)u^{\mathrm{tot}}(y)\big)\,\dy,
\end{multline*}
where the normal $\nu$ on $\partial D\cup \partial B_\varepsilon(x)$ is directed out of $D\setminus B_\varepsilon(x)$.
Since $x\notin\operatorname{supp}\chi$, standard computations (see for instance~\cite[Theorem 3.1]{ColKre}) yield, as $\varepsilon\to0$,
\begin{multline*}
u^{\mathrm{tot}}(x)+\int_{\partial D}\left[\frac{\partial\Phi(x,y)}{\partial\nu(y)}u^{\mathrm{tot}}(y)-\Phi(x,y)\frac{\partial u^{\mathrm{tot}}(y)}{\partial\nu}\right]\mathrm d\sigma(y)\\
=\int_{D}\Phi(x,y)\big(2\nabla\chi(y)\cdot\nabla u^{\mathrm{tot}}(y)+\Delta\chi(y)u^{\mathrm{tot}}(y)\big)\,\dy.
\end{multline*}
Similarly, applying Green's second identity to $u^{\mathrm{inc}}$ and $\Phi(x,\cdot)$ on $D$ yields
\begin{equation*}
-u^{\mathrm{inc}}(x)=\int_{\partial D}\left[\frac{\partial\Phi(x,y)}{\partial\nu(y)}u^{\mathrm{inc}}(y)-\Phi(x,y)\frac{\partial u^{\mathrm{inc}}}{\partial\nu}(y)\right]\mathrm d\sigma(y),
\end{equation*}
while applying \cite[Thm.~3.3]{ColKre} to $u$ on $\R^n\setminus\overline D$ yields
\begin{equation*}
0=\int_{\partial D}\left[\frac{\partial\Phi(x,y)}{\partial\nu(y)}u(y)-\Phi(x,y)\frac{\partial u}{\partial\nu}(y)\right]\mathrm d\sigma(y).
\end{equation*}
Altogether,
\begin{equation*}
u^{\mathrm{tot}}(x)=u^{\mathrm{inc}}(x)+\int_{D}\Phi(x,y)\big(2\nabla\chi(y)\cdot\nabla u^{\mathrm{tot}}(y)+\Delta\chi(y)u^{\mathrm{tot}}(y)\big)\,\dy.
\end{equation*}
On the other hand, let $\chi_1\in C^\infty_{0,\operatorname{supp}\chi}$ with $x\notin\operatorname{supp}\chi_1$.
Then, by a weak version of Green's second identity applied to $\chi u^{\mathrm{tot}}$ and $\chi_1\Phi(x,\cdot)$ on $\Omega$ (i.e., applying Green's first identity~\eqref{Eq:Green} twice), it holds that
\begin{equation*}
\int_{D}\Phi(x,y)\big(2\nabla\chi(y)\cdot\nabla u^{\mathrm{tot}}(y)+\Delta\chi(y)u^{\mathrm{tot}}(y)\big)\,\dy=\langle\partial_\nu u^{\mathrm{tot}},\mathrm{Tr}_\Gamma\overline{\Phi(x,\cdot)}\rangle_{\Bcal^*(\Gamma)\times\Bcal(\Gamma)},
\end{equation*}
since $\mathrm{Tr}_\Gamma \overline{u^{\mathrm{tot}}}=0$, hence
\begin{equation*}
u^{\mathrm{tot}}(x)=u^{\mathrm{inc}}(x)+\langle\partial_\nu u^{\mathrm{tot}},\mathrm{Tr}_\Gamma\overline{\Phi(x,\cdot)}\rangle_{\Bcal^*(\Gamma)\times\Bcal(\Gamma)}.
\end{equation*}
In other words 
 the scattered field is given by
\begin{equation*}
u(x)=\langle\partial_\nu u^{\mathrm{tot}},\mathrm{Tr}_\Gamma(\overline{\chi \Phi(x,\cdot)})\rangle_{\Bcal^*(\Gamma)\times\Bcal(\Gamma)}.
\end{equation*}
This equation  can be rewritten, using~\eqref{eq:Aphi2} and \eqref{eq:adjointy}, as
\begin{equation}\label{Eq:Link-ddn-ie}
u(x)=\langle \mathrm{Tr}^*_\Gamma\partial_\nu u^{\mathrm{tot}},\overline{\chi \Phi(x,\cdot)}\rangle=\Acal(\mathrm{tr}^*_\Gamma(\partial_\nu u^{\mathrm{tot}}))(x).
\end{equation}
Thus $u$ and $\cA(\mathrm{tr}^*_\Gamma(\partial_\nu u^{\mathrm{tot}}))$ coincide on $D\cap \Omega$.
Since both satisfy the Helmholtz equation on $\Omega$, which is connected, they coincide on $\Omega$ (e.g., \cite[p.~72]{ColKre}), i.e., $u=\cA(\mathrm{tr}^*_\Gamma(\partial_\nu u^{\mathrm{tot}}))|_\Omega$. But, by Theorem \ref{thm:trace}, since $k^2\not\in \sigma(-\Delta_D(\Omega_-))$, \eqref{eq:iemainB} has a unique solution $f\in \cB^*(\Gamma)$ and $u=\cA\mathrm{tr}^*_\Gamma f|_\Omega$. It follows from \eqref{eq:Tracesame} that $\Trc_\Gamma \cA\mathrm{tr}^*_\Gamma f=\Trc_\Gamma \cA\mathrm{tr}^*_\Gamma \partial_\nu u^{\mathrm{tot}}$, i.e., that $\sA f = \sA \partial_\nu u^{\mathrm{tot}}$, so that $f=\partial_\nu u^{\mathrm{tot}}$. 
\end{proof}

We finish this section by making an extension to the case of general compact $\Gamma$ satisfying \eqref{eq:GamRest}. To relate the solution $f\in \cB^*(\Gamma)$ to $\partial_\nu u^{\mathrm{tot}}\in \cB^*(\partial \Omega)$ in this more general case we embed $\cB^*(\partial \Omega)$ in $\cB^*(\Gamma)$. The embedding $j:\cB^*(\partial \Omega)\to \cB^*(\Gamma)$ is given by $j := (\Trcr^*_{\Gamma})^{-1}\Trcr^*_{\partial \Omega}$, so that, if $f\in \cB^*(\partial \Omega)$ and $\varphi \in \cB(\Gamma)$, in which case  $\varphi = \Trc_\Gamma u$ for some $u\in H^1(\R^n)$, then, by \eqref{eq:adjointy},
$$
\langle jf, \varphi\rangle_{\cB^*(\Gamma)\times \cB(\Gamma)} = \langle f, \Trc_{\partial \Omega} u\rangle_{\cB^*(\partial \Omega)\times \cB(\partial \Omega)}.
$$
Recall that $\Omega_-$ is given by \eqref{eq:Omegam}.
\begin{theorem} \label{cor:normal}
Assume $\Omega$ is an extension domain and that $k^2\not\in \sigma(-\Delta_D(\Omega_-))$, where $\Omega_-=O\setminus \Gamma$.
Then the unique solution $f\in\Bcal^*(\Gamma)$ to the operator equation~\eqref{eq:iemainB} can be expressed in terms of the total field $u^{\mathrm{tot}}$ as
\begin{equation*}
f=j (\partial_\nu u^{\mathrm{tot}}).
\end{equation*}
\end{theorem}
\begin{proof} Arguing as in the proof of Proposition \ref{prop:normal}, the solution to the scattering problem can be written as $u=\cA(\mathrm{tr}^*_{\partial \Omega}(\partial_\nu u^{\mathrm{tot}}))|_\Omega$.  Arguing as at the end of the proof of Proposition \ref{prop:normal} we have, by Theorem \ref{thm:trace}, since $k^2\not\in \sigma(-\Delta_D(\Omega_-))$, that \eqref{eq:iemainB} has a unique solution $f\in \cB^*(\Gamma)$,  $u=(\cA\mathrm{tr}^*_\Gamma f)|_\Omega$, and, by \eqref{eq:Tracesame}, that $\Trc_{\partial \Omega} \cA\mathrm{tr}^*_\Gamma f=\Trc_{\partial \Omega} \cA\mathrm{tr}^*_{\partial \Omega} \partial_\nu u^{\mathrm{tot}}$. Further, by Theorem \ref{thm:trace} and Corollary \ref{cor:IEfinal}, $\Trcr^*_\Gamma f$ is the unique solution to \eqref{eq:iemain2} and $\Trcr^*_\Gamma f\in H^{-1}_{\partial \Omega}$. Thus, where $\sA_{\partial \Omega}:= \Trc_{\partial \Omega} \cA \Trcr^*_{\partial \Omega}$, i.e., $\sA_{\partial \Omega}$ denotes the operator $\sA$ of \eqref{eq:sAdef} when $\Gamma=\partial \Omega$, we have that $\sA_{\partial \Omega} (\Trcr^*_{\partial \Omega})^{-1}\Trcr^*_\Gamma f = \sA_{\partial \Omega} \partial_\nu u^{\mathrm{tot}}$. Thus, 
 if $k^2\not\in \sigma(-\Delta_D(\intt(O)))$,  so that $\sA_{\partial \Omega}$ is invertible by Theorem \ref{thm:trace}, it follows that $(\Trcr^*_{\partial \Omega})^{-1}\Trcr^*_\Gamma f = \partial_\nu u^{\mathrm{tot}}$, i.e., $f = j(\partial_\nu u^{\mathrm{tot}})$. 

To show that this result holds for all $k>0$ with $k^2\not\in \sigma(-\Delta_D(\Omega_-))$ we use the following limiting argument. Suppose $k_0>0$ and $k^2_0\in \sigma(-\Delta_D(\intt(O))) \setminus \sigma(-\Delta_D(\Omega_-))$. Choose $\varepsilon>0$ such that, where $\widetilde I_\varepsilon := (k^2_0-\varepsilon,k^2_0+\epsilon)$, it holds that $\widetilde I_\epsilon \cap \sigma(-\Delta_D(\intt(O)))=\{k^2_0\}$ and  $\widetilde I_\varepsilon \cap \sigma(-\Delta_D(\Omega_-))=\emptyset$, and let $I_\varepsilon := \{k>0:k^2\in \widetilde I_\varepsilon\}$. Denote $\cA$ and $\sA$ temporarily by $\cA_k$ and $\sA_k$ to indicate their dependence on $k$. Then, by Theorem \ref{thm:trace}, $\sA_k$ is invertible for $k\in I_\varepsilon$. 

By \eqref{eq:inv}, $\cA_k = -R(k^2)$, where $R(z):=(-\Delta -z)^{-1}$ is the resolvent of the Laplacian. For $\chi\in C_0^\infty(\R^n)$ the so-called {\em cut-off resolvent}, $\chi R(k^2)\chi:L^2(\R^n)\to H^2(\R^n)$, depends continuously in the norm topology  on $k$, for $k\in \C$ with $\mathrm{Re}\, k>0$ and $\mathrm{Im}\, k\geq 0$, by standard limiting absorption arguments (see \cite[Thm.~4.1]{Agmon75} and its proof). It follows  by an interpolation argument (see, e.g., the proof of \cite[Prop.~2.1]{SiavashSimon2}) that $\chi \cA_k \chi:H^{-1}(\R^n)\to H^1(\R^n)$ also depends continuously on $k$ in the norm topology, for every $\chi\in C_0^\infty(\R^n)$, in particular for
$\chi\in C_{0,\Gamma}^\infty$.  Thus $\chi \cA_k:H^{-1}_\Gamma \to H^1(\R^n)$ depends continuously on $k$ in the norm topology, so that the same is true for
$\sA_k=\Trc_\Gamma \chi \cA_k\Trcr^*_\Gamma$. Since $\sA_k$ is invertible for $k\in I_\varepsilon$,  also $\sA_k^{-1}$ depends continuously on $k$ in the norm topology, for $k\in I_\varepsilon$.

For the case that $k=k_0$, let $u^{\mathrm{inc}}$, $u$, and $u^{\mathrm{tot}}$ denote the incident, scattered, and total fields, respectively, and let $f\in \cB^*(\Gamma)$ denote the unique solution of \eqref{eq:iemainB}. Choose $u^{\mathrm{inc}}_k\in \Hol(\R^n)$, for $k\in I_\varepsilon$, so that $u^{\mathrm{inc}}_{k_0} = u^{\mathrm{inc}}$, so that $u^{\mathrm{inc}}_k$ depends continuously in $\Hol(\R^n)$ on $k\in I_\epsilon$, and such that, for some fixed domain $D\supset O$,  $(\Delta + k^2)u_k^i = 0$ in $D$. Then, for $k\in I_\varepsilon$, $h_k:= -\Trc_\Gamma u^{\mathrm{inc}}_k$ depends continuously on $k$ in $\cB(\Gamma)$, so that $f_k:=\sA_k^{-1}h_k$ also depends continuously on $k$, and note that $f_{k_0}=f$. Further, for $k\in I_\varepsilon$, $u^{\mathrm{tot}}_k:= u^{\mathrm{inc}}_k+\cA_k\Trcr^*_\Gamma f_k$ depends continuously on $k$ in $\Hol(\R^n)$. Thus it follows from  \eqref{Eq:Norm-der} that $\partial_\nu u^{\mathrm{tot}}_k$ depends continuously on $k$ in $\cB^*(\partial \Omega)$. But note that $u^{\mathrm{tot}}=u^{\mathrm{tot}}_{k_0}$, by Theorem \ref{thm:trace}, so that $\partial_\nu u^{\mathrm{tot}} = \partial_\nu u^{\mathrm{tot}}_{k_0}$. Further, by the first part of the proof and Theorem \ref{thm:trace} we have that $f_k=j(\partial_\nu u^{\mathrm{tot}}_k)$ for $k\in I_\varepsilon \setminus \{k_0\}$. It follows, by taking the limit $k\to k_0$, that the same holds for $k=k_0$, i.e., that $f=j(\partial_\nu u^{\mathrm{tot}})$. 
\end{proof}

\section{Reformulation  as an integral equation} \label{sec:ie}

In this section we show, as \S\ref{sec:Aie}, that, if $\Gamma$ is the support of a Radon measure that satisfies certain assumptions, then the equation  \eqref{eq:sAdef} on $\Gamma$ can be rewritten, equivalently, as a integral equation on $\Gamma$, where the integration is with respect to the measure $\mu$. A crucial role in our arguments will be played by a second trace operator, $\widetilde \Trc_\Gamma$, that we introduce in \S\ref{sec:trL2}. The distinction between this trace operator and the trace operator $\Trc_\Gamma$ introduced in \S\ref{sec:trace} is subtle: for $u\in \cS(\R^n)$ both $\widetilde \Trc_\Gamma u$ and $\Trc_\Gamma u$ are equivalence classes of functions defined on $\Gamma$, but the equivalence relation is different; equality $\mu$-a.e.\ for $\widetilde \Trc_\Gamma$,  equality q.e.\ for $\Trc_\Gamma$. In \S\ref{sec:emb} we introduce Assumption \ref{ass:Gabriel}, under which these two equivalence relations coincide. (Conditions on $\Gamma$ and $\mu$ under which this holds have been established recently in \cite{Hinz}.) This assumption allows us in embed the space $\cB(\Gamma)$ of \S\ref{sec:trace} into $L^2(\Gamma,\mu)$, as made clear in Proposition \ref{prop:equiv}, that fleshes out and extends results announced as \cite[Corollary 5.1]{Hinz} that will be key to proving convergence, in the next section, of our Galerkin method. 

The function spaces we use through this section are recalled in \S\ref{sec:fs} or have been introduced in \S\ref{sec:trace}. The properties of measures we make use of are recalled in 
\S\ref{sec:meas}.

\subsection{Traces to $L^2(\Gamma,\mu)$} \label{sec:trL2}

Suppose that $\mu$ is a Radon measure whose support is the compact set $\Gamma$.  Define a second trace operator $\widetilde \Trc_\Gamma:\cS(\R^n)\to L^2(\Gamma,\mu)$ by
\begin{equation} \label{eq:tr2}
\widetilde \Trc_\Gamma v := v|_\Gamma, \qquad v\in \cS(\R^n),
\end{equation}
noting that this is a continuous linear operator which has dense range (since the range is dense in $C(\Gamma)$ and $C(\Gamma)$ is dense in $L^2(\Gamma,\mu)$ by Lusin's theorem (e.g., \cite[Thm.~2.24]{Rudin}), as $\mu$ is Radon). (We are omitting some details here; the image $\widetilde \Trc_\Gamma v$ is not precisely $v|_\Gamma$ but is the equivalence class in $L^2(\Gamma,\mu)$ that has $v|_\Gamma$ as an element.) Let $C_0(\R^n)$ denote the space of continuous functions that vanish at infinity, a Banach space equipped with the supremum norm that we denote by $\|\cdot\|_{C_0(\R^n)}$. Clearly, 
$$
\|\widetilde \Trc_\Gamma v\|_{L^2(\Gamma,\mu)} \leq (\mu(\Gamma))^{1/2}\|v\|_{C_0(\R^n)}, \qquad u\in \cS(\R^n),
$$
so that, since $\cS(\R^n)$ is dense in $C_0(\R^n)$, $\widetilde \Trc_\Gamma$ has a unique extension (which we will denote again by $\widetilde \Trc_\Gamma$) that is a continuous linear operator $C_0(\R^n)\to L^2(\Gamma,\mu)$ that satisfies \eqref{eq:tr2}, indeed this continuous extension is given by \eqref{eq:tr2} for $v\in C_0(\R^n)$.

We will make, for much of the rest of the paper, the following assumption on $\Gamma$ and $\mu$.

\begin{assumption} \label{ass:cont}
The Radon measure $\mu$ whose support is the compact set $\Gamma$ is such that, for some constant $C>0$,
\begin{equation} \label{eq:bd}
\|\widetilde \Trc_\Gamma v\|_{L^2(\Gamma,\mu)} \leq C\|v\|_{H^1(\R^n)}, \qquad v\in \cS(\R^n).
\end{equation}
\end{assumption}
Clearly (cf.~\cite[\S18.5]{Triebel97FracSpec}), if Assumption \ref{ass:cont} holds, then $\widetilde \Trc_\Gamma$ has a unique extension (which we will again denote by $\widetilde \Trc_\Gamma$) that is a continuous linear operator $H^1(\R^n)\to L^2(\Gamma,\mu)$ with the property that \eqref{eq:tr2} holds. (Again, we are omitting some details; we are understanding  $\cS(\R^n)$ as embedded (densely) in $H^1(\R^n)$ in the usual way, so are identifying $v\in \cS(\R^n)$ with the equivalence class in $H^1(\R^n)$ that contains $v$.) 
Further, if Assumption \ref{ass:cont} holds, then, for every measurable $E\subset \Gamma$, 
\begin{equation} \label{eq:capmu}
\mathrm{cap}(E)=0 \Rightarrow \mu(E)=0.
\end{equation}
To see this, suppose $E\subset \Gamma$ is measurable with $\mu(E)>0$ but $\mathrm{cap}(E)=0$. As $\mu$ is a regular Borel measure, there exists a compact  set $K\subset E$ with $\mu(K)>0$ and $\mathrm{cap}(K)=0$. By the definition \eqref{eq:capdef} there exists a sequence $(u_m)_{m\in \N}\subset \cS(\R^n)$ such that, for each $m$, $u_m\geq 1$ on $K$ so that $\|\widetilde \Trc_\Gamma u_m\|_{L^2(\Gamma,\mu)}\geq (\mu(K))^{1/2}$, while also $\|u_m\|_{H^1(\R^n)}\to 0$ as $m\to\infty$, contradicting \eqref{eq:bd}. From \cite[Thm.~7.2.1]{adams2012function}, and an inspection of its proof, we have, in fact, the following stronger result.
\begin{theorem}
Assumption \ref{ass:cont} holds if and only if, for some constant $c>0$,
\begin{equation} \label{eq:charact}
\mu(K) \leq c^2\, \mathrm{cap}(K), \qquad \mbox{for all compact } K\subset \Gamma.
\end{equation}
Further, the least possible values for the constants $c$ in \eqref{eq:charact} and $C$ in \eqref{eq:bd} are equivalent; precisely, there exists $c_n\geq 1$, depending only on the dimension $n$, such that, if $C>0$ is such that \eqref{eq:bd} holds, then \eqref{eq:charact} holds with $c:= C$, while if \eqref{eq:charact} holds for some $c>0$, then \eqref{eq:bd} holds with $C:=  c_n c$.
\end{theorem}
\noindent See \cite[Thm.~7.2.1]{adams2012function} for other equivalent characterisations of Assumption \ref{ass:cont}. Note that if $\Gamma$ is non-empty, so that $\mu(\Gamma)>0$, and Assumption \ref{ass:cont} holds,  so that $\capp(\Gamma)>0$ by \eqref{eq:capmu}, then  $H_\Gamma^{-1}\neq \{0\}$, by Remark \ref{rem:zero}.

The following theorem, a special case of \cite[Thm.~1]{Jonsson79}, provides a large class of measures that satisfy Assumption \ref{ass:cont}. Recall that $\mu$ is upper $d$-regular if it satisfies \eqref{eq:udr}.

\begin{theorem} \label{thm:Bie} If $\mu$ is a Radon measure whose support is the compact set $\Gamma$ and $\mu$ is upper $d$-regular on $\Gamma$, for some $d\in (n-2,n]$, then $\mu$ and $\Gamma$ satisfy Assumption \ref{ass:cont}.
\end{theorem}
\begin{proof}
Apply  \cite[Thm.~1]{Jonsson79} with (in the notation of that theorem) $k=0$, $\alpha=1$, and $\beta=1-(n-d)/2$ with $0<d<n$, noting that if $\mu$ is upper $d$-regular with $d=n$ then it is upper $d$-regular for $0<d<n$.
\end{proof}

\subsection{Embedding $\cB(\Gamma) $ in $L^2(\Gamma,\mu)$} \label{sec:emb}

In the case that Assumption \ref{ass:cont} holds, let us define a mapping $\iota:\cB(\Gamma)\to L^2(\Gamma,\mu)$ by
\begin{equation} \label{eq:idef}
\iota := \widetilde \Trc_\Gamma \circ E_\Gamma,
\end{equation}
where the isometry $E_\Gamma:\cB(\Gamma)\to H^1(\R^n)$ is the right inverse of $\Trc_\Gamma$ defined in \S\ref{sec:trace}. (The various function spaces and the mappings between them are shown in Figure \ref{fig:cd}.) Note that, since $\|E_\Gamma\|=1$, 
$$
\|\iota\| \leq C,
$$
where $C$ is the constant in Assumption \ref{ass:cont}.
Note also that $\widetilde \Trc_\Gamma v=0$ if $v\in \widetilde H^1(\Gamma^c)$,  by \eqref{eq:tr2}, the density of $\cS(\R^n)$ in $H^1(\R^n)$, and the continuity of $\widetilde \Trc_\Gamma:H^1(\R^n)\to L^2(\Gamma,\mu)$, so that $\widetilde H^1(\Gamma^c)\subset \ker \widetilde \Trc_\Gamma$. Further, if  $f\in V_1(\Gamma^c)$, then $E_\Gamma\circ \Trc_\Gamma f= E_\Gamma\circ \Trcr_\Gamma f=f$. Thus, and recalling \eqref{eq:ker} and \eqref{eq:perp}, 
\begin{equation} \label{eq:Traces}
\widetilde \Trc_\Gamma = \iota \circ \Trc_\Gamma.
\end{equation}
We note also that $\widetilde \Trc_\Gamma v = v|_\Gamma$ for $v \in H^1(\R^n)\cap C_0(\R^n)$, indeed 
\begin{equation} \label{eq:TrE}
\widetilde \Trc_\Gamma v = \tilde v|_\Gamma, \qquad v\in H^1(\R^n),
\end{equation}
where $\tilde v$ is any quasi-continuous representative of $v$, in particular the representative defined by \eqref{eq:qcrep}.
To see this, note that \eqref{eq:TrE} holds with $v$ replaced by $v_r\in H^1(\R^n)\cap C_0(\R^n)$, where $v_r$ is defined by \eqref{eq:qcrep}, and then take the limit $r\to 0^+$, noting that, as discussed below \eqref{eq:qcrep}, $v_r\to \tilde v$ in $H^1(\R^n)$, so that $\widetilde \Trc_\Gamma v_r=v_r|_\Gamma\to \widetilde \Trc_\Gamma v$ in $L^2(\Gamma,\mu)$, so that (e.g., \cite[Thm.~3.12]{Rudin}) there exists a sequence $r_n\to 0$ such that   $v_{r_n}|_\Gamma\to \widetilde \Trc_\Gamma v$ $\mu$-a.e. But also, as noted below \eqref{eq:qcrep}, $v_r|_\Gamma\to \tilde v|_\Gamma$, as $r\to 0^+$, q.e.\ on $\Gamma$, and so, by \eqref{eq:capmu}, also $\mu$-a.e., so \eqref{eq:TrE} holds.

If $f\in \cB(\Gamma)$ and $\tilde f$ is a representative of $f$, it follows from \eqref{eq:idef} and \eqref{eq:TrE} that $\iota f$ is the equivalence class in $L^2(\Gamma,\mu)$ of functions that are equal $\mu$-a.e.\ to $\tilde f$. 
Let
\begin{equation} \label{eq:Hdef}
\bH(\Gamma,\mu) := \iota(\cB(\Gamma))
\end{equation}
denote the range of $\iota$, which, by \eqref{eq:Traces}, is also the range of $\widetilde \Trc_\Gamma$ since $\Trc_\Gamma$ is surjective, and let us equip $\bH(\Gamma,\mu)$ with the norm (cf.\ \eqref{eq:Bnorm})
\begin{equation} \label{eq:Hnorm}
\|f\|_{\bH(\Gamma,\mu)} := \min \{ \|v\|_{H^1(\R^n)} : v\in H^1(\R^n) \mbox{ and } \widetilde \Trc_F v=f\} = \min \{ \|g\|_{\cB(\Gamma)} : g\in \cB(\Gamma) \mbox{ and } \iota g=f\},
\end{equation}
with which $\bH(\Gamma,\mu)$ is a Hilbert space. Define $\iota_\bH:\cB(\Gamma)\to \bH(\Gamma,\mu)$ by $\iota_\bH f = \iota f$, $f\in \cB(\Gamma)$, so that, by definition, $\iota_\bH$ is surjective and is continuous with $\|\iota_\bH\|\leq 1$. We will see shortly that $\iota$ is injective, so that $\iota$ is an embedding and $\iota_\bH$ is a unitary isomorphism, if and only if the following additional assumption on $\mu$ and $\Gamma$ holds. Recall from \S\ref{sec:fs} that, for $u\in H^1(\R^n)$, $\tilde u$ denotes a quasi-continuous representative of $u$, and that any two quasi-continuous representatives are equal q.e. 
 
 \begin{assumption} \label{ass:Gabriel} For every $u\in H^1(\R^n)$,
 $$
 \tilde u = 0 \;\, \mu\mbox{-a.e. on } \Gamma\ \Leftrightarrow\ \tilde u=0 \mbox{ q.e. on } \Gamma.
 $$
 \end{assumption} 
\noindent  Note that if Assumption \ref{ass:cont} is satisfied then, by \eqref{eq:capmu}, the left implication ($\Leftarrow$) holds in the above assumption.

That the above assumption is satisfied in the case that $\Gamma$ is a $d$-set, for $n-2<d<n$, follows from  \cite[Prop.~6.7]{CaetanoJFA} (which result establishes a version of Assumption \ref{ass:Gabriel} for general Bessel-potential and Besov spaces). Cases beyond that of the $d$-set have been investigated recently in \cite{Hinz}, which paper is concerned precisely with establishing conditions that ensure that Assumption \ref{ass:Gabriel} holds, and that the analogous result holds with $H^1(\R^n)$ replaced by some Bessel-potential space. The results established there imply, in particular, the following result (see \cite[Cor.~2.7]{Hinz} and also \cite[Thms.~2.3, 2.9]{Hinz}).

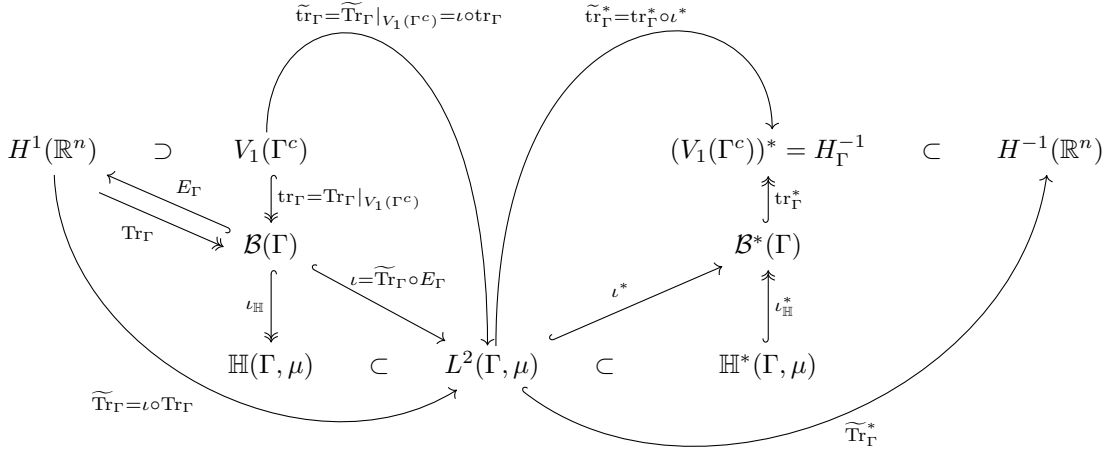
\begin{figure}

\begin{tikzcd}[sep=small]
	{H^1(\R^n)} & \supset & {V_1(\Gamma^c)} &&&& {(V_1(\Gamma^c))^*=H^{-1}_\Gamma} & \subset & {H^{-1}(\R^n)} \\
	\\
	&& {\cB(\Gamma)} &&&& {\cB^*(\Gamma)} \\
	\\
	\\
	&& {\mathbb{H}(\Gamma,\mu)} & \subset & {L^2(\Gamma,\mu)} & \subset & {\mathbb{H}^*(\Gamma,\mu)}
	\arrow["{\Trc_\Gamma}"', shift right=3, two heads, from=1-1, to=3-3]
	\arrow["{E_\Gamma}"', hook, from=3-3, to=1-1]
	\arrow["{\Trcr_\Gamma=\Trc_\Gamma|_{V_1(\Gamma^c)}}", hook, two heads, from=1-3, to=3-3]
	\arrow["{\iota_{\mathbb{H}}}"', hook, two heads, from=3-3, to=6-3]
	\arrow["{\iota =\widetilde \Trc_\Gamma\circ E_\Gamma}"{yshift=4,xshift=-13}, hook, from=3-3, to=6-5]
	\arrow["{\Trcr_\Gamma^*}"', hook, two heads, from=3-7, to=1-7]
	\arrow["{\iota^*}",hook,from=6-5, to=3-7]
	\arrow["{\iota_{\mathbb{H}}^*}"', hook, two heads,  from=6-7, to=3-7]
	\arrow["{\widetilde \Trc_\Gamma=\iota\circ\Trc_\Gamma}"', bend right=60,from=1-1,to=6-5]
	\arrow["{\widetilde \Trcr_\Gamma=\widetilde \Trc_\Gamma|_{V_1(\Gamma^c)}=\iota\circ \Trcr_\Gamma}" {yshift=2pt,xshift=-33pt}, shift right, controls={+(0,2) and +(0,5)},from=1-3,to=6-5]
	\arrow["{\widetilde \Trc_\Gamma^*}"',bend right=60,hook,from=6-5,to=1-9]
	\arrow["{\widetilde \Trcr^*_\Gamma}=\Trcr^*_\Gamma\circ \iota^*"{yshift=2pt,xshift=23pt},shift right, controls={+(0,5) and +(0,2)},from=6-5,to=1-7]
\end{tikzcd}
\caption{A commutative diagram showing the key function spaces and the mappings between them, assuming that Assumptions \ref{ass:cont} and \ref{ass:Gabriel} hold. All mappings are continuous. The mappings $\hookrightarrow$, $\twoheadrightarrow$, $\hookdoubleheadrightarrow$ are, additionally, injective, surjective, and a unitary isomorphism, respectively. $A\subset B$, for Hilbert spaces $A$ and $B$, indicates that $A$ is a subset of $B$ with continuous embedding.}
\label{fig:cd}
\end{figure}

\begin{theorem} \label{thm:summ}
Suppose that $J\in \N$ and, for $j=1,\ldots,J$, that $a_j>0$, $n-2<d_j\leq n$, and the compact set $\Gamma_j\subset \R^n$ is a $d_j$-set. Let
\begin{equation} \label{eq:musum}
\mu := \sum_{j=1}^J a_j \cH^{d_j}|_{\Gamma_j} \quad \mbox{and} \quad \Gamma := \bigcup_{j=1}^J \Gamma_j = \supp(\mu).
\end{equation}
Then $\mu$ and $\Gamma$ satisfy Assumptions \ref{ass:cont} and \ref{ass:Gabriel}.
\end{theorem} 
\begin{proof} That Assumption \ref{ass:cont} is satisfied follows from Theorem \ref{thm:Bie}, since $\mu$ is $d$-upper regular, for $d=\min\{d_1,\ldots,d_J\}$, as discussed below \eqref{eq:udr}. That Assumption \ref{ass:Gabriel} holds follows from \cite[Cor.~2.7]{Hinz}. 
\end{proof}

Let $\widetilde \Trcr_\Gamma := \widetilde \Trc_\Gamma|_{V_1(\Gamma^c)}$, so that $\widetilde \Trcr_\Gamma = \iota \circ \Trcr_\Gamma$, and, identifying $L^2(\Gamma,\mu)$ with its dual, let $\widetilde \Trcr_\Gamma^*:L^2(\Gamma,\mu)\to H_\Gamma^{-1}$ and $\widetilde \Trc_\Gamma^*:L^2(\Gamma,\mu)\to H^{-1}(\R^n)$ denote the adjoints of $\widetilde \Trcr_\Gamma$ and  $\widetilde \Trc_\Gamma$, respectively, so that
\begin{equation} \label{eq:adj}
( \phi,\widetilde\Trc_\Gamma u )_{L^2(\Gamma,\mu)} = \langle  \widetilde \Trc_\Gamma^*\phi,u \rangle, \qquad \phi\in L^2(\Gamma,\mu), \;u\in H^1(\R^n),
\end{equation}
and $\widetilde \Trcr_F^*\phi=\widetilde \Trc_\Gamma^*\phi$, $\phi\in L^2(\Gamma,\mu)$. Similarly, let $\iota^*:L^2(\Gamma,\mu)\to \cB^*(\Gamma)$ denote the adjoint of $\iota$.
The following result provides the promised characterisation of the injectivity of $\iota$.

\begin{proposition} \label{prop:equiv}
Suppose that $\Gamma$ and $\mu$ satisfy Assumption \ref{ass:cont}. Then the following are equivalent:
\begin{enumerate}
\item[(i)] Assumption \ref{ass:Gabriel} holds;
\item[(ii)] $\iota$ is injective;
\item[(iii)] the range of $\iota^*$ is dense in $\cB^*(\Gamma)$;
\item[(iv)] $\ker \widetilde \Trc_\Gamma= \widetilde H^1(\Gamma^c)$;
\item[(v)] $\iota_\bH$ is a unitary isomorphism;
\item[(vi)] $\widetilde \Trcr_\Gamma$ is injective;
\item[(vii)] the ranges of $\widetilde \Trcr_\Gamma^*$ and $\widetilde \Trc_\Gamma^*$ are dense in $H_\Gamma^{-1}$.
\end{enumerate}
\end{proposition}
\begin{proof} 
If $f\in \cB(\Gamma)$, so that $f=\tilde v|_\Gamma$ q.e. on $\Gamma$, for some $v\in H^1(\R^n)$, where $\tilde v$ is any quasi-continuous representative of $v$, and $\iota f=0$, then $\tilde v = 0$ $\mu$-a.e. on $\Gamma$, so that also $\tilde v=0$ q.e. on $\Gamma$ if Assumption \ref{ass:Gabriel} holds, so that $f=0$. Thus (i)$\Rightarrow$(ii). 

If $\iota$ is not injective then there exists $v\in H^1(\R^n)$ with quasi-continuous representative $\tilde v$ such that $f=\Trc_\Gamma v=\tilde v|_\Gamma$ is non-zero on some set of positive capacity on $\Gamma$ but $\iota f=\tilde v|_\Gamma =0$ $\mu$-a.e. on $\Gamma$, contradicting Assumption \ref{ass:Gabriel}. Thus (ii)$\Rightarrow$(i). 

It is an elementary result that if $A_j$ is a vector space, for $j=1,2,3$, $F:A_1\to A_2$ and $G:A_2\to A_3$ are linear mappings, and $F$ is surjective, then $\ker F=\ker (G\circ F)$ if and only if $G$ is injective. Applying this result with $A_1=H^1(\R^n)$, $A_2=\cB(\Gamma)$, $A_3=L^2(\Gamma,\mu)$, $F=\Trc_\Gamma$ and $G= \iota$, so that, by \eqref{eq:Traces}, $G\circ F = \widetilde \Trc_\Gamma$, we see that (ii)$\Leftrightarrow$(iv), since $\ker   \Trc_\Gamma = \widetilde H^1(\Gamma^c)$ by \eqref{eq:ker}. Since $\widetilde \Trcr_\Gamma = \iota \circ \Trcr_\Gamma$ and $\Trcr_\Gamma$ is injective, the same argument shows that (ii)$\Leftrightarrow$(vi). 

The mappings $\widetilde \Trcr_\Gamma^*$ and $\widetilde \Trc_\Gamma^*$ have the same range. Since $\widetilde \Trcr_\Gamma$ is the adjoint of $\widetilde \Trcr_\Gamma^*$, by a standard property of adjoint maps \cite[Thm.~4.12]{Rudin2} the range of $\widetilde \Trcr_\Gamma^*$ is dense in  $H_\Gamma^{-1}$ if and only if $\widetilde \Trcr_\Gamma$ is injective. Thus (vi)$\Leftrightarrow$(vii). Similarly, since $\iota$ is the adjoint of $\iota^*$, (ii)$\Leftrightarrow$(iii). 

Finally, if $\iota$, and so also $\iota_\bH$, is not injective, then $\iota_\bH$ is neither an isometry nor an isomorphism. But if $\iota$ is injective then $\iota_\bH$ is bijective, so that, if $g\in \cB(\Gamma)$ and $f=\iota_\bH g$, then $g$ is the unique element $h\in \cB(\Gamma)$ with $\iota h= f$ so that, by \eqref{eq:Hnorm}, $\|f\|_{\bH(\Gamma,\mu)} = \|g\|_{\cB(\Gamma)}$, and $\iota_\bH$ is a unitary isomorphism. Thus  (ii)$\Leftrightarrow$(v). 
\end{proof}

\begin{remark} \label{rem:embedding} In the case that Assumptions \ref{ass:cont} and \ref{ass:Gabriel} hold, it follows from the discussion above, including the above proposition, that $\iota:\cB(\Gamma)\to L^2(\Gamma,\mu)$ is an embedding (i.e., a continuous, injective linear mapping), moreover with dense range, so that, through the embedding $\iota$, $\cB(\Gamma)$ is densely embedded in $L^2(\Gamma,\mu)$. Further, through $\iota^*:L^2(\Gamma,\mu)\to \cB^*(\Gamma)$, $L^2(\Gamma,\mu)$ is densely embedded in $\cB^*(\Gamma)$, so that $(\cB(\Gamma),L^2(\Gamma,\mu),\cB^*(\Gamma))$ forms a Gelfand triple. 
\end{remark}

Let $\bH^*(\Gamma,\mu)$ denote the dual space of $\bH(\Gamma,\mu)$ 
and note that (since we identify $L^2(\Gamma,\mu)$ with its dual space) $\bH(\Gamma,\mu)\subset L^2(\Gamma,\mu)\subset \bH^*(\Gamma,\mu)$, and these embeddings are continuous with dense range, as long as Assumption \ref{ass:cont} holds, i.e., $(\bH(\Gamma,\mu),L^2(\Gamma,\mu), \bH^*(\Gamma,\mu))$ is a Gelfand triple.  Let $\iota^*_\bH:\bH^*(\Gamma,\mu)\to \cB^*(\Gamma)$ denote the adjoint of $\iota_\bH$, so that $\iota^*=\iota_\bH^*|_{L^2(\Gamma,\mu)}$, and define $\bA:\bH^*(\Gamma,\mu)\to \bH(\Gamma,\mu)$ by
\begin{equation} \label{eq:bAdef}
\bA := \iota_\bH \sA \iota_\bH^* = \iota_\bH \Trcr_\Gamma A \Trcr_\Gamma^*\iota_\bH^* = \widetilde \Trc_\Gamma A \Trcr_\Gamma^*\iota_\bH^*.
\end{equation}
Since, by Proposition \ref{prop:equiv}, $\iota_\bH$ (and so also $\iota_\bH^*$) are unitary isomorphisms if Assumptions \ref{ass:cont} and \ref{ass:Gabriel} hold, the following corollary follows from Theorem \ref{thm:trace}.
\begin{corollary} \label{cor:iefinal} Suppose $\Gamma$ satisfies \eqref{eq:GamRest}.  If Assumptions \ref{ass:cont} and \ref{ass:Gabriel} hold, then $\bA$ is unitarily equivalent to $\sA$ and $A$, so that $\bA$ is Fredholm of index zero and is invertible if and only if  $k^2\not \in \SigO$. Further, if $\phi\in H_\Gamma^{-1}$ and $f:= (\iota_\bH^*)^{-1}(\Trcr_\Gamma^*)^{-1}\phi\in \bH^*(\Gamma,\mu)$, then $\phi$ satisfies \eqref{eq:iemain2} if and only if $f$ satisfies
\begin{equation} \label{eq:iefinal}
\bA f = \mathfrak{g} := \iota_\bH \Trcr_\Gamma g = \widetilde \Trc_\Gamma g = -\widetilde \Trc_\Gamma(\chi u^{\mathrm{inc}}) = -u^{\mathrm{inc}}|_\Gamma. 
\end{equation}
Moreover, if  $f\in \bH^*(\Gamma,\mu)$ satisfies \eqref{eq:iefinal} and $u=\cA\Trcr_\Gamma^*\iota_\bH^* f$, then $u|_\Omega$ is the unique solution of the exterior sound-soft scattering problem.
\end{corollary}

\subsection{Writing $\bA$ as an integral operator} \label{sec:Aie}

Related to the last sentence of Corollary \ref{cor:iefinal}, the next lemma 
characterises $\cA\phi$ as an integral with respect to the measure $\mu$ when Assumption \ref{ass:cont} holds and, for some $f\in L^2(\Gamma,\mu)\subset \bH^*(\Gamma,\mu)$, $\phi=\Trcr_\Gamma^*\iota_\bH^* f$, in which case, since $\widetilde \Trcr_\Gamma = \iota\circ \Trcr_\Gamma$ and $\iota_\bH^* f = \iota^* f$,  $\phi=\widetilde \Trcr_\Gamma^*f = \widetilde \Trc_\Gamma^*f$. Note that the integral in \eqref{eq:intrep1} is well-defined for all $x\in \Gamma^c$ and $f\in L^2(\Gamma,\mu)$ (indeed for $f\in L^1(\Gamma,\mu)$) since $\Phi(x,\cdot)\in C(\Gamma)$, so that $\Phi(x,\cdot)f$ is $\mu$-integrable on $\Gamma$.

\begin{lemma} \label{lem:Aref} If $\Gamma$ and $\mu$ satisfy Assumption \ref{ass:cont} and $f\in L^2(\Gamma,\mu)$, then
\begin{equation} \label{eq:intrep1}
\cA\Trcr_\Gamma^*\iota_\bH^* f(x) =  \cA\widetilde \Trc_\Gamma^* f(x) = \int_\Gamma \Phi(x,y)f(y)\, \rd \mu(y), \qquad  x\in \Gamma^c.
\end{equation}
\end{lemma}
\begin{proof} Suppose $\sigma\in C_{0,\Gamma}^\infty(\R^n)$ and $\phi\in H^{-1}_\Gamma$. Then, by \eqref{eq:Aphi2},
\begin{equation} \label{eq:rep}
\cA \phi(x) = \langle\phi,\overline{\sigma \Phi(x,\cdot)}\rangle, \qquad  x\not\in \supp(\sigma).
\end{equation}
Thus, if $f\in L^2(\Gamma,\mu)$, so that $\widetilde \Trc^*_\Gamma f\in H_\Gamma^{-1}$, 
$$
\cA \widetilde \Trc^*_\Gamma f(x) = \langle\widetilde \Trc^*_\Gamma f,\overline{\sigma \Phi(x,\cdot)}\rangle = (f,\widetilde \Trc_\Gamma(\overline{\sigma \Phi(x,\cdot)}))_{L^2(\Gamma,\mu)} = \int_\Gamma \Phi(x,y)f(y)\, \rd \mu(y), \quad  x\not\in \supp(\sigma).
$$
As this holds for every $\sigma\in C_{0,\Gamma}^\infty(\R^n)$, the result follows.
\end{proof}

We noted above that the integral in \eqref{eq:intrep1} is well-defined for $x\in \Gamma^c$. It is also well-defined when $x\in \Gamma$ under  additional constraints on $f$ and $\mu$. In the following assumption we use the notation
$$
f_n(r) := \left\{\begin{array}{ll}
|\log(r)|, & n=2,\\
r^{2-n}, & n> 2.
\end{array}\right.
$$

\begin{assumption} \label{ass:int} 
It holds that
$$
\sup_{x\in \R^n} \int_{B_r(x)}f_n(|x-y|)\rd\mu(y) \to 0 \quad \mbox{as} \quad r\to 0^+.
$$
\end{assumption} 

\noindent The relevance of this assumption to our integrals is that it follows from \eqref{eq:Phidef} and the power series expansion for $H_\nu^{(1)}$ (e.g., \cite[\S10.8]{NIST}) that Assumption \ref{ass:int} holds if and only if
\begin{equation} \label{eq:aint2}
\sup_{x\in \R^n} \int_{B_r(x)}|\Phi(x,y)|\rd\mu(y) \to 0 \quad \mbox{as} \quad r\to 0^+.
\end{equation}

Like Assumptions \ref{ass:cont} and \ref{ass:Gabriel}, the above assumption holds for a large class of measures. The following lemma is an immediate consequence of Lemma \ref{lem:ub}, noting that if $\mu$ is upper $d$-regular on $\Gamma$ then it is upper $d$-regular on $B_R(0)$, for every $R>0$. 

\begin{lemma} \label{lem:growth} If $\mu$ is a Radon measure whose support is the compact set $\Gamma$ and $\mu$ is upper $d$-regular on $\Gamma$, for some $d\in (n-2,n]$, then $\mu$ and $\Gamma$ satisfy Assumption \ref{ass:int}.
\end{lemma}
\begin{proof}
Apply Lemma \ref{lem:ub} with $\tau=f_n$, if $n> 2$, and with $\tau(r):= f_2(r)$, for $0<r\leq 1$, $\tau(r):=0$, for $r>1$, if $n=2$.
\end{proof}


The following characterisation interprets $\bA f$ as an integral when $f\in L^\infty(\Gamma,\mu)$. Note that, from \eqref{eq:bAdef}, the definition \eqref{eq:Akdef} of $A$, and the comments above Lemma \ref{lem:Aref},
\begin{equation} \label{eq:bArep}
\bA f = \widetilde \Trc_\Gamma A \widetilde \Trc^*_\Gamma f = \widetilde \Trc_\Gamma (\chi \cA \widetilde \Trc^*_\Gamma f), \qquad f\in L^2(\Gamma,\mu).
\end{equation}

\begin{theorem} \label{prop:interep} If $\Gamma$ and $\mu$ satisfy Assumptions \ref{ass:cont} and \ref{ass:int} and $f\in L^\infty(\Gamma,\mu)$, then $\cA\widetilde \Trc_\Gamma^*f\in H^{1,\mathrm{loc}}(\R^n)\cap C(\R^n)$, and the right hand side of \eqref{eq:intrep1} provides a continuous representative of $\cA\widetilde \Trc_\Gamma^*f$, so that
\begin{equation} \label{eq:intrep2}
\bA f(x) =  \int_\Gamma \Phi(x,y)f(y)\, \rd \mu(y), \qquad \mbox{$\mu$-a.e. } x\in \Gamma.
\end{equation}
\end{theorem}
\begin{proof} Suppose $\Gamma$ and $\mu$ satisfy Assumptions \ref{ass:cont} and \ref{ass:int}, so that they also satisfy \eqref{eq:aint2}, and $f\in L^\infty(\Gamma,\mu)$. 

We show first that the right hand side of \eqref{eq:intrep2} is continuous on $\R^n$, arguing in part as in the proof\footnote{There is a small inaccuracy in that proof which we correct. It is not correct that (44) in \cite[Prop.~4.5]{Caetano24} holds with the constant indicated for all $x\in \R^{n+1}$ (it does hold with a larger value for that constant). Cf.~our footnote \ref{foot:8}.} of \cite[Prop.~4.5]{Caetano24} (and cf.~\cite[Lem. 2.18]{Kral}). For $\varepsilon>0$, let  
$$
\Phi_\epsilon(x,y) := \left\{\begin{array}{ll} \Phi(x,y), & |x-y|>\varepsilon,\\ \Phi(0,\varepsilon \hat e), & |x-y|\leq \varepsilon,\end{array}\right.
$$
where $\hat e\in \R^n$ is any fixed unit vector, so that $\Phi_\epsilon \in C(\R^n\times \R^n)$, and let
\begin{equation} \label{eq:Fdef}
F(x) := \int_\Gamma \Phi(x,y)f(y)\, \rd \mu(y), \qquad F_\varepsilon(x) := \int_\Gamma \Phi_\varepsilon(x,y)f(y)\, \rd \mu(y), \qquad x\in \R^n,
\end{equation}
noting that $F(x)$ is well-defined by \eqref{eq:Fdef} for every $x\in \R^n$ by \eqref{eq:aint2}, and that, since $\Phi_\varepsilon$ is continuous, $F_\epsilon \in C(\R^n)$.  
Note that $|\Phi(x,y)| = |\Phi(0,r\hat e)|$, where $r=|x-y|$, and that $|\Phi(0,r\hat e)|$ is a decreasing function of $r$ on $(0,\infty)$, since the same is true for $|H_\nu^{(1)}(r)|$, for $\nu\geq 0$ \cite[\S13.74]{Watson}.  It follows that, for $x\in \R^n$,
 $$
 |F(x)-F_\epsilon(x)| \leq 2 \|f\|_{L^\infty(\Gamma,\mu)}\int_{B_\varepsilon(x)}|\Phi(x,y)|\rd\mu(y) \to 0 \quad \mbox{as} \quad \epsilon\to 0^+,
 $$
 uniformly on $\R^n$, by \eqref{eq:aint2}, so that $F\in C(\R^n)$. We have $\cA\widetilde \Trc_\Gamma^*f\in H^{1,\mathrm{loc}}(\R^n)$ by \eqref{eq:mapping}. It remains to show that $\cA\widetilde \Trc_\Gamma^*f(x)= F(x)$, for a.e.\ $x\in \R^n$, so that $\cA\widetilde \Trc_\Gamma^*f\in H^{1,\mathrm{loc}}(\R^n)\cap C(\R^n)$, and to show that \eqref{eq:intrep2} holds.

By the Lebesgue-Radon-Nikodym theorem (e.g., \cite[Thm.~6.10]{Rudin}), 
$$
\mu = \mu_1+\mu_2,
$$
where $\mu_1$ and $\mu_2$ are positive measures,  the measure $\mu_2$ and $n$-dimensional Lebesgue measure are mutually singular, so that $\Gamma_s :=\supp(\mu_2)\subset \Gamma$ has Lebesgue measure zero, 
and, for some non-negative $f_\mu\in L^1(\R^n)$ with $\supp(f_\mu)\subset \Gamma$,
$$
\mu_1(E) = \int_E f_\mu \, \rd x,
$$
for every measurable $E\subset \R^n$. Since $\mu$ is Radon and satisfies Assumptions \ref{ass:cont} and \ref{ass:int}, the same holds for $\mu_1$ and $\mu_2$, so that $\widetilde \Trc_\Gamma:H^1(\R^n)\to L^2(\Gamma,\mu_j)$ is continuous, for $j=1,2$. It follows that
$$
\widetilde \Trc^*_\Gamma  f= \widetilde \Trc^*_{\Gamma,1} f + \widetilde \Trc^*_{\Gamma,2} f,
$$
where  $\widetilde \Trc^*_{\Gamma,j}:L^2(\Gamma,\mu_j)\to H^{-1}(\R^n)$ denotes the adjoint of $\widetilde \Trc_\Gamma:H^1(\R^n)\to L^2(\Gamma,\mu_j)$, given by \eqref{eq:adj} with $\mu$ replaced by $\mu_j$, $j=1,2$. Let $F_j$ be defined, for $j=1,2$, by \eqref{eq:Fdef} with $F$ and $\mu$ replaced by $F_j$ and $\mu_j$, and note that $F=F_1+F_2$.

Lemma \ref{lem:Aref} implies that $\cA\widetilde \Trc^*_{\Gamma,2}f(x) = F_2(x)$, for a.e.\ $x\in \Gamma_s^c$, so for a.e.\ $x\in \R^n$. Further, it follows from \eqref{eq:adj} that
$$
\langle \widetilde \Trc^*_{\Gamma,1}f,u\rangle = (f,\widetilde \Trc_\Gamma u)_{L^2(\Gamma,\mu_1)} = \int_{\R^n} \bar u f f_\mu  \, \rd x, \qquad u\in \cS(\R^n),
$$
so that $\Trc^*_{\Gamma,1}f =ff_\mu \in L^1_{\mathrm{comp}}(\R^n)$. Thus, by \eqref{eq:Newt}, for a.e.\ $x\in \R^n$,
$$
\cA\widetilde \Trc^*_{\Gamma,1}f(x) = \int_{\R^n} \Phi(x,y)f(y)f_\mu(y) \, \rd y = F_1(x).
$$
Thus $\cA\widetilde \Trc^*_{\Gamma}f(x) = F(x)$ for a.e.\ $x\in \R^n$. Since $F\in C(\R^n)$, and so is a quasi-continuous representative of $\cA\widetilde \Trc^*_{\Gamma}f$, it follows from \eqref{eq:TrE} that $\widetilde \Trc_\Gamma (\chi \cA \widetilde \Trc^*_\Gamma f) = F|_\Gamma$, so that \eqref{eq:intrep2} follows from \eqref{eq:bArep}.
\end{proof}

\begin{remark}[Comparison with known results] \label{rem:kn1}
In the case that $\Gamma$ is a $d$-set and $\mu$ is $\cH^d$, i.e. $d$-dimensional Hausdorff measure, this result and Lemma \ref{lem:Aref} reproduce \cite[Thm.~3.16]{caetano2023integral}. In the classical case that the obstacle $O$ is the closure of a bounded Lipschitz domain and $\Gamma=\partial O$ is its boundary, $\Gamma$ is a $d$-set with $d=n-1$ and $\cH^d$ coincides with  surface measure on $\Gamma$ \cite[Thm.~3.8]{EvansGariepy}, so that \eqref{eq:bArep} is the usual representation for the standard  single-layer potential operator,   e.g., \cite[Eqn.~(7.8)]{mclean2000strongly}, \cite[Thm.~3.3.5]{SaSc:11}.
\end{remark}

\section{Galerkin methods for the integral equation formulation} \label{sec:GM}

In this section we discuss Galerkin methods for the solution of the operator equation \eqref{eq:iemain2}, which can be written equivalently as \eqref{eq:iefinal}, with $\bA$ characterised as an integral operator  in Theorem \ref{prop:interep}. 

\subsection{Galerkin methods and their convergence} \label{sec:GM1}

The standard Galerkin method procedure, applied to solve the operator equation \eqref{eq:iemain2} written in variational form as \eqref{eq:iemainV}, is as follows. To find approximations to the  solution $\phi\in V:= H_\Gamma^{-1}$ of \eqref{eq:iemainV} one chooses a sequence of finite-dimensional subspaces $(V_N)_{N\in \N}\subset V$ and, for each $N$, seeks a {\em Galerkin approximation} $\phi_N\in V_N$ to $\phi$ that satisfies 
\begin{equation} \label{eq:GM}
a(\phi_N,\psi_N) = \langle g,\psi_N \rangle, \qquad \psi_N\in V_N.
\end{equation}
Recalling  from Proposition \ref{prop:ie1} that the solution to the exterior sound-soft scattering problem is $u|_\Omega$, where $u=\cA\phi$, we can compute from $\phi_N$ an approximation to $u$ defined by $u_N:= \cA\phi_N|_\Omega$.
In the case that  \eqref{eq:iemain2} is uniquely solvable we say that {\em the Galerkin method converges as $N\to\infty$} if, for some $N_0\in \N$, \eqref{eq:GM} has a unique solution $\phi_N\in V_N$ for $N\geq N_0$ and $\phi_n\to \phi$ as $N\to \infty$.
Proposition \ref{prop:coer}, that establishes that $a(\cdot,\cdot)$ is continuous and compactly perturbed coercive, together with Theorem \ref{thm:invert} that characterises invertibility, combined with standard convergence theory for the Galerkin method (e.g., \cite[\S4.2.3, \S4.2.5]{SaSc:11}), give the following convergence result (cf.~\cite[Thm.~5.6]{Caetano24}, \cite[Thm.~4.3, 4.4]{caetano2023integral}).

\begin{proposition} \label{prop:GM} 
Suppose $\Gamma$ satisfies  \eqref{eq:GamRest} and $k^2\not\in \SigO$, so that \eqref{eq:iemain2} has a unique solution $\phi\in V$ for every $g\in V_1(\Gamma^c)$. Then the Galerkin method converges as $N\to\infty$  for every $g\in V_1(\Gamma^c)$ if and only if
\begin{equation} \label{eq:ad}
e_N(\psi) := \min_{\psi_N\in V_N} \|\psi-\psi_N\|_{H^{-1}(\R^n)} \to 0 \quad \mbox{as} \quad N\to \infty, \quad \mbox{for all }\psi\in V,
\end{equation} 
in which case also $u_N(x)\to u(x)$ as $N\to\infty$ for every $x\in \Gamma^c$.
Further, if \eqref{eq:ad} holds then, for some constants $C_G>0$ and $N_0\in \N$ independent of $g$ and $x\in \Omega$,
$$
\|\phi-\phi_N\|_{H^{-1}(\R^n)} \leq C_G\, e_N(\phi), \quad |u(x)-u_N(x)| \leq C_G \, e_N(\phi)e_N(\zeta_x), \qquad N\geq N_0, \;\; x\in \Omega,
$$
where $\zeta_x\in H^{-1}_\Gamma$ is the unique solution to \eqref{eq:iemainV} with $\phi$, $g$, replaced by $\zeta_x$, $\chi_x \Phi(x,\cdot)$, respectively, for some $\chi_x\in C_{0,O}^\infty$ with $x\not\in \supp(\chi_x)$.  
\end{proposition}

In the case that Assumption \ref{ass:cont} holds, so that $\widetilde \Trc_\Gamma:H^1(\R^n)\to L^2(\Gamma,\mu)$ and $\widetilde \Trc^*_\Gamma: L^2(\Gamma,\mu)\to H^{-1}(\R^n)$, 
 a way to create an approximation space $V_N$ is to construct a finite-element type approximation space $\widetilde V_N\subset L^2(\Gamma,\mu)$ and set $V_N:=  \widetilde\Trc^*_\Gamma(\widetilde V_N)$. Following \cite{caetano2023integral,Caetano24} we will illustrate this for the simplest case of a piecewise-constant finite element approximation. For $N\in \N$ let $M_N\in \N$ and let $\tau_N=\{T_{j,N}\}_{j=1}^{M_N}$ be a {\em mesh} of $\Gamma$, a collection of $\mu$-measurable subsets of $\Gamma$ (the {\em elements}) such that\footnote{\label{foot:12}We will later use \eqref{eq:mesh}, with $\Gamma$ replaced by $\Gamma'$, as the definition, more generally, of a mesh on some $\mu$-measurable $\Gamma'\subset \Gamma$.}
\begin{equation} \label{eq:mesh}
\mu\left(\Gamma \setminus\bigcup_{j=1}^{M_N} T_{j,N}\right) =0, \quad \mu(T_{j,N})>0 \mbox{ for } j=1,\ldots,M_N, \quad \mbox{and } \mu(T_{j,N}\cap T_{j',N})=0 \mbox{ for } j\neq j',
\end{equation}
and let $h_N:= \max_{j=1,\ldots,M_N} \mathrm{diam}(T_{j,N})$. 
For $N\in \N$ let
\begin{equation} \label{eq:tilVN}
\widetilde V_N := \{f\in L^\infty(\Gamma,\mu): f|_{T_{j,N}}=c_{j} \mbox{ on } T_{j,N}, \mbox{ for some }c_j\in \C, j=1,\ldots,M_N\}\subset L^2(\Gamma,\mu),
\end{equation}
and set $V_N:= \widetilde \Trc^*_\Gamma(\widetilde V_N)$. We assume that the sequence of {\em finite element subspaces} $\widetilde V_N$ is chosen so that $h_N\to 0$ as $N\to\infty$, in which case it follows that
$$
\tilde e_N(f) := \min_{f_N\in \widetilde V_N} \|f-f_N\|_{L^2(\Gamma,\mu)} \to 0 \quad \mbox{as} \quad N\to \infty, \quad \mbox{for all }f\in L^2(\Gamma,\mu).
$$
For certainly this holds for $f\in C(\Gamma)$, since $\Gamma$ is compact, and $C(\Gamma)$ is dense in $L^2(\Gamma,\mu)$ by Lusin's theorem (e.g., \cite[Thm.~2.24]{Rudin}) as $\mu$ is a Radon measure. The following result is thus a corollary of Propositions \ref{prop:GM} and \ref{prop:equiv} (cf.~\cite[Cor.~5.2]{Hinz}).

\begin{corollary} \label{cor:GM}
Suppose that $\Gamma$ satisfies  \eqref{eq:GamRest}, $k^2\not\in \SigO$, $\Gamma$ and $\mu$ satisfy Assumption \ref{ass:cont},  $V_N:= \widetilde \Trc^*_\Gamma(\widetilde V_N)$, for $N\in \N$, with $\widetilde V_N$ given by \eqref{eq:tilVN}, and $h_N\to0$ as $N\to\infty$.  Then \eqref{eq:ad} holds and the Galerkin method converges as $N\to\infty$  for every $g\in V_1(\Gamma^c)$ if and only $\Gamma$ and $\mu$ satisfy Assumption \ref{ass:Gabriel}.
\end{corollary}
\begin{proof} Since $h_N\to 0$ implies that $\tilde e_N(f)\to 0$ as $N\to\infty$, it is clear that \eqref{eq:ad} holds if and only if the range of $\widetilde \Trc^*_\Gamma$ is dense in $H^{-1}(\R^n)$ which, by Proposition \ref{prop:equiv}, is equivalent to Assumption \ref{ass:Gabriel}. Thus the corollary follows from Proposition \ref{prop:GM}.
\end{proof} 

Continuing to assume that Assumption \ref{ass:cont} holds, let $\{f_{i,N}\}_{i=1}^{M_N}\subset L^\infty(\Gamma,\mu)$ be a basis for $\widetilde V_N$; one obvious choice, that is an orthonormal basis in $L^2(\Gamma,\mu)$,  is
\begin{equation} \label{eq:basis}
f_{i,N} := \frac{\mathbf{1}_{T_{i,N}}}{(\mu(T_{i,N}))^{1/2}}, \qquad i=1,\ldots,M_N,
\end{equation}
where $\mathbf{1}_{T_{i,N}}$ denotes the characteristic function of $T_{i,N}$. Let $\{e_{i,N} = \widetilde \Trc^*_\Gamma f_{i,N}\}_{i=1}^{M_N}$ denote the corresponding basis for $V_N$. If $\phi_N\in V_N$, then
\begin{equation} \label{eq:phiN}
\phi_N= \sum_{i=1}^{M_N} c_{i,N} e_{i,N},
\end{equation}
for some coefficients $c_{1,N},\ldots,c_{M_N,N}\in \C$, and $\phi_N$ satisfies \eqref{eq:GM} if and only if
\begin{equation} \label{eq:GM2}
\sum_{j=1}^{M_N} a_{i,j}^N c_{j,N} = b_{i,N}, \qquad i=1,\ldots,M_N,
\end{equation}
where
\begin{equation} \label{eq:aij}
a_{i,j}^N := a(e_{j,N},e_{i,N}) = \langle A e_{j,n},e_{i,N}\rangle, \quad b_{i,N} := \langle g,e_{i,N}\rangle, \qquad i,j=1,\ldots,M_N.
\end{equation}
Under certain assumptions these Galerkin method coefficients can be written explicitly as integrals with respect to the measure $\mu$.
\begin{proposition} \label{prop:coeff}
If $\Gamma$ and $\mu$ satisfy Assumption \ref{ass:cont} then, where $\mathfrak{g}$ is given by \eqref{eq:iefinal},
\begin{equation} \label{eq:aij2}
a_{i,j}^N = ( \bA f_{j,n},f_{i,N})_{L^2(\Gamma,\mu)}, \quad b_{i,N} = ( \mathfrak{g},f_{i,N})_{L^2(\Gamma,\mu)}=-\int_{\Gamma} u^{\mathrm{inc}}  \overline{f_{i,n}}\, \rd \mu, \qquad i,j=1,\ldots,M_N,
\end{equation}
 and 
\begin{equation} \label{eq:uNrep}
u_N(x) = \sum_{j=1}^{M_N} c_{j,N} \int_\Gamma \Phi(x,y) f_{j,N}(y)\, \rd \mu(y), \qquad x\in \Omega.
\end{equation}
If also  Assumption \ref{ass:int} holds, then
\begin{equation} \label{eq:aij3}
a_{i,j}^N= \int_{\Gamma} \int_{\Gamma}\Phi(x,y)f_{j,n}(y)\overline{f_{i,n}(x)} \,\rd \mu(y)\rd \mu(x), \qquad i,j=1,\ldots,M_N.
\end{equation}
\end{proposition}
\begin{proof} Equations \eqref{eq:aij2} follow from \eqref{eq:aij},  \eqref{eq:adj}, and \eqref{eq:bArep}. Since $u_N= \cA\phi_N|_\Omega$,  \eqref{eq:uNrep} follows from Lemma \ref{lem:Aref}.
Further, if Assumption \ref{ass:int} holds, then, by Theorem \ref{prop:interep}, $\bA f_{j,n}$ is given explicitly by  \eqref{eq:intrep2}, so that \eqref{eq:aij3} follows.
\end{proof}
\noindent We  discuss the computation of the integrals in the above proposition for particular classes of $\Gamma$, $\mu$, and the mesh $\tau_N$, in Remark \ref{rem:ni}.

From this point on we restrict attention to the case where $\Gamma$ and $\mu$ are as in Theorem \ref{thm:summ}. In that case, as stated in Theorem \ref{thm:summ}, Assumptions \ref{ass:cont} and \ref{ass:Gabriel} are satisfied and, as noted in the proof of that theorem, $\mu$ is upper $d$-regular for some $d\in (n-2,n]$, so that Assumption \ref{ass:int} is satisfied, by Lemma \ref{lem:growth}. Thus the following result follows from Propositions \ref{prop:GM} and \ref{prop:coeff}, Corollary \ref{cor:GM}, and the above discussion.

\begin{corollary} \label{cor:GM2}
Suppose that $\Gamma$ satisfies  \eqref{eq:GamRest}, $k^2\not\in \SigO$, and $\mu$ is as in Theorem \ref{thm:summ}. Then there exists $h>0$ and $C_G>0$ such that, if $\tau_N$ is a mesh with $h_N\leq h$ and  $V_N:= \widetilde \Trc^*_\Gamma(\widetilde V_N)$ with $\widetilde V_N$ given by \eqref{eq:tilVN}, then the Galerkin method equations \eqref{eq:GM} and \eqref{eq:GM2} have exactly one solution and $\|\phi-\phi_N\|_{H^{-1}(\R^n)} \leq C_G\, e_N(\phi)$,  $|u(x)-u_N(x)| \leq C_G \, e_N(\phi)e_N(\zeta_x)$, for $x\in \Omega$, where $\zeta_x$ is as in Proposition \ref{prop:GM}. Further, $e_N(\phi)\to 0$ and $e_N(\zeta_x)\to 0$ as $h_N\to 0$, and the coefficients in \eqref{eq:GM2} are given explicitly as integrals with respect to the measure $\mu$ in \eqref{eq:aij2} and \eqref{eq:aij3}.
\end{corollary} 

\begin{remark}[Comparison with previous results and with BEM] \label{rem:kn2}
The simplest application of Corollary \ref{cor:GM2} is to the case where $\Gamma$ is a $d$-set and $\mu$ is (a multiple of) $\cH^d$. In that case Corollary \ref{cor:GM2} reproduces \cite[Eqns.~(4.6), (4.7), Thm.~4.3]{caetano2023integral}; in particular, equations (4.6) and (4.7) in \cite{caetano2023integral} coincide with \eqref{eq:aij2} and \eqref{eq:aij3}. A particular instance is the case where $O$ is the closure of a bounded Lipschitz domain and $\Gamma=\partial O$ is its boundary, which is a $d$-set with $d=n-1$. In that case, as noted in Remark \ref{rem:kn1}, $\cH^d$ coincides with  surface measure on $\Gamma$, so that \eqref{eq:GM2}, \eqref{eq:aij2}, and  \eqref{eq:aij3} are the equations defining a classical piecewise-constant Galerkin boundary element method (BEM) applied to the equation $S\phi=-u^{\mathrm{inc}}|_\Gamma$, where $S$ is the standard (acoustic) single-layer boundary integral operator.
\end{remark}

\subsection{The case that $\intt(O)$ is an $n$-set} \label{sec:onset}
As a more sophisticated application of Corollary \ref{cor:GM2}, consider the case that the interior of the obstacle,  $\intt(O)$, is a non-empty open $n$-set. By this we mean that  $\intt(O)$ satisfies the so-called {\em measure density condition} (e.g., \cite{hajlasz_sobolev_2008}), that, for some constant $c>0$,
$$
|\intt(O)\cap B(x,r)| \geq cr^n, \qquad x\in \intt(O),
$$
in which case (see \cite[Prop.~1 on p.~205]{JW84}) $F:= \overline{\intt(O)}$ is an $n$-set and $|\partial F|=0$. Suppose also that
the  boundary\footnote{This analysis extends in an obvious way, with a straightforward extension of Corollary \ref{cor:GM3}, to the case that $\partial O$ is the union of finitely many compact sets $\gamma_i$, with each $\gamma_i$ a $d_i$-set, for some $d_i\in (n-2,n]$, and $\mu(\gamma_i\cap\gamma_j)=0$, for $i\neq j$.} $\partial O$ is a $d$-set for some $d\in [n-1,n)$. For example, this is the case, with $F=O$ and $d=n-1$, if $O$ is the closure of a bounded Lipschitz domain. It is also the case, with $F=O$, if $O$ is a member of the family of  ``classical snowflakes'', that includes the standard Koch snowflake,  which  is studied in \cite[\S5.1]{CaetanoJFA}. For if $O$ is such a snowflake then $\partial O$ is a $d$-set with $d\in (1,2)$, by \cite[Prop.~5.3, Rem.~5.4]{CaetanoJFA}, and $\intt(O)$ is a non-empty open $n$-set as shown in \cite[\S5.3]{CaetanoJFA} (our  open $n$-set is termed an {\em interior regular domain} in \cite{CaetanoJFA}). An example where $F\neq O$  is the obstacle $O=B_1(0)\cup\{(x_1,0):|x_1|\leq 2\}$, for which $F=B_1(0)$ and $\partial O$ is a $d$-set with $d=n-1$.

Set $\Gamma=O$ (so that $\Omega_-=\SigO=\emptyset$), and, for some constants $a_n$, $a_d>0$, set
\begin{equation} \label{eq:muo}
\mu := a_d\cH^d|_{\partial O} + a_n \cH^n|_F,
\end{equation}
so that $\supp(\mu)=F\cup \partial O = \Gamma$. 
Where $h=(h_a,h_n)$, with $h_a$, $h_n>0$, construct a mesh $\tau_h$ on $\Gamma$ that consists of a mesh $\tau_h^n:= \{T^n_j\}_{j=1}^{N_n}$ on $\intt(O)$ in the sense of \eqref{eq:mesh}  and footnote \ref{foot:12}, with $\mathrm{diam}(T^n_j)\leq h_n$, $j=1,\ldots,N_n$ (note that $\mu=a_n \cH^n$ on $\intt(O)$), and a mesh $\tau_h^d:=\{T^d_j\}_{j=1}^{N_d}$ on $\partial O$, with  $\mathrm{diam}(T^d_j)\leq h_d$, $j=1,\ldots,N_d$ (note that $\mu=a_d\cH^d$ on $\partial O$).  The union of these two meshes is $\tau_h := \{T^n_j\}_{j=1}^{N_n}\cup \{T^d_j\}_{j=1}^{N_d}$. Let (cf.~\eqref{eq:tilVN})
\begin{eqnarray} \label{eq:Vh}
\widetilde V_h &:=& \{f\in L^\infty(\Gamma,\mu): f|_T \mbox{ is constant $\mu$-a.e. on $T$, for $T\in \tau_h$}\},\\ \label{eq:Vhd}
\widetilde V_h^d &:=& \{f\in L^\infty(\partial O,\mu): f|_T \mbox{ is constant $\mu$-a.e. on $T$, for $T\in \tau^d_h$}\}\\ \nonumber
&=& \{f\in \widetilde V_h: f=0 \mbox{ on } \intt(O)\},
\end{eqnarray}
and let $V_h:= \widetilde \Trc^*_\Gamma(\widetilde V_h)$ and $V^d_h:= \widetilde \Trc^*_\Gamma(\widetilde V^d_h)$. With these notations and definitions the following result is a consequence of Corollary \ref{cor:GM2}. The point of this result is that, as long as the meshes on $\intt(O)$ and $\partial O$ are sufficiently refined ($h_n$ and $h_d$ are both $\leq h_0$), convergence of the Galerkin method is obtained, for all $k>0$, by refining only the mesh on $\partial O$; see Remark \ref{remkn3} for further discussion.

\begin{corollary} \label{cor:GM3}
Suppose that $O$, $\Gamma$, $V_h$, and $V_h^d$ are as defined above. Then, for every $k>0$ there exists $h_0>0$ and $C_G>0$ such that, if $h_d\leq h_0$ and $h_n\leq h_0$, then the Galerkin method equation
$$
a(\phi_h,\psi_h)=\langle g,\psi_h\rangle, \qquad \psi_h\in V_h,
$$
has a unique solution $\phi_h\in V_h$ that satisfies
\begin{equation} \label{eq:ee}
\|\phi-\phi_h\|_{H^{-1}(\R^n)} \leq C_G\, e_{h_d}(\phi) \quad \mbox{and} \quad |u(x)-u_h(x)| \leq C_G \, e_{h_d}(\phi)e_{h_d}(\zeta_x), \quad \mbox{ for } x\in \Omega, 
\end{equation}
where $u_h:= \cA\phi_h|_\Omega$,  $\zeta_x$ is as in Proposition \ref{prop:GM}, and
$$
e_{h_d}(\psi) := \min_{\psi_h\in V_h^d} \|\psi-\psi_h\|_{H^{-1}(\R^n)}, \qquad \psi\in H^{-1}_\Gamma.
$$
Further, $e_{h_d}(\phi)\to 0$ and $e_{h_d}(\zeta_x)\to 0$, for every $x\in \Omega$, as $h_d\to 0$.
\end{corollary}
\begin{proof}
The first sentence is immediate from Corollary \ref{cor:GM2}. The second  sentence follows since $e_{h_d}(\psi)\to 0$ as $h\to 0$ for every $\psi\in H^{-1}_{\partial O}$, and $\phi, \zeta_x\in H^{-1}_{\partial O}$ by Corollary \ref{cor:IEfinal}. (Note that, provided we choose, as in the definition of $\zeta_x$ in Proposition \ref{prop:GM}, $\chi_x\in C_{0,O}^\infty$ so that $x\not\in \supp(\chi_x)$, it follows that  $\zeta_x$ is the solution of \eqref{eq:iemain2} with $u^{\mathrm{inc}}:= -\chi_x \Phi(x,\cdot)\in H^1(\R^n)$.)
\end{proof}

\begin{remark}\label{remkn3} 
For the case considered in Corollary \ref{cor:GM3}, an alternative integral equation formulation and Galerkin method is to take $\Gamma=\partial O$, so that $\Gamma$ is a $d$-set, and $\mu=\cH^d$, discretising $\Gamma$ with the approximation space $V_h^d$. (In the case that $\Omega$ is Lipschitz so that $d=n-1$, this is a standard Galerkin BEM as noted in Remark \ref{rem:kn2}.) This leads to a Galerkin method approximation $\phi_h^d$ which, by Corollary \ref{cor:GM2} (or \cite[Thm.~4.3]{caetano2023integral}), satisfies the same error estimate \eqref{eq:ee} (possibly with different constants),  using a smaller approximation space (the space $V_h^d$ has dimension $N_d$ while $V_h$ has dimension $N_d+N_n$). But, with the choice $\Gamma = \partial O$, the integral equation is uniquely solvable and the error estimate \eqref{eq:ee} holds  {\em only for} $k^2\not \in \sigma(-\Delta_D(\mathrm{int}(O)))$, and the (infinite) set of eigenvalues $\sigma(-\Delta_D(\mathrm{int}(O)))$ is, in general, unknown. In contrast, Corollary \ref{cor:GM3} guarantees convergence   for all $k>0$ using the Galerkin method based on the space $V_h$. Moreover, Corollary \ref{cor:GM3} shows that convergence is achieved with this approximation space by taking $N_d\to\infty$ with $N_n$ fixed (but sufficiently large), in which regime $V_h^d$ and $V_h$ have asymptotically the same dimension. Recall also our discussion of the choice of $\Gamma$, and its implications for the solvability of the integral equation, in Remark \ref{rem:choice}.
\end{remark}

\subsection{Solution regularity} \label{sec:reg}

In this section, as a precursor to proving convergence rates for our Galerkin methods, we establish regularity results for the solution of \eqref{eq:iemain2}. These results apply when $\Gamma$ and $\mu$ are given  by \eqref{eq:musum}, together  with additional assumptions, including  that
\begin{equation} \label{eq:reg}
\Gamma_i\cap \Gamma_j = \emptyset, \qquad i,j=1,\ldots,J, \quad i\neq j.
\end{equation}
The regularity results that we derive depend on earlier regularity results from \cite{Caetano24,caetano2023integral} for the special case that $\Gamma$ is a $d$-set, for some  $d\in (n-2,n]$, and $\mu=\cH^d$ (in other words, the case that $\Gamma$ and $\mu$ are given  by \eqref{eq:musum} with $J=1$). 

Suppose first of all that $\Gamma$ is a $d$-set, for some  $d\in (n-2,n]$, and $\mu=\cH^d$. Then (see \cite[Thm.~1]{Jonsson79}, \cite[Thm.~18.6]{Triebel97FracSpec}, or \cite[\S2.4]{Caetano24}), for $s>(n-d)/2$ the trace operator  $\widetilde \Trc_\Gamma:\cS(\R^n)\to L^2(\Gamma,\mu)$ has a unique extension that is a continuous linear operator $H^s(\R^n)\to L^2(\Gamma,\cH^d)$. Varying slightly the notation\footnote{In \cite{Caetano24,caetano2023integral} our $\bH^t(\Gamma,\cH^d)$ is denoted more briefly as $\bH^t(\Gamma)$.} of \cite{Caetano24,caetano2023integral} we denote the range of this operator by $\bH^t(\Gamma,\cH^d):= \widetilde \Trc_\Gamma(H^s(\R^n))$, where
$$
t:=s-\frac{n-d}{2},
$$
and equip  $\bH^t(\Gamma,\cH^d)$ with the quotient norm 
\begin{equation} \label{eq:Hsnorm}
\|f\|_{\bH^t(\Gamma,\cH^d)} := \min \{ \|v\|_{H^s(\R^n)} : v\in H^s(\R^n) \mbox{ and } \widetilde \Trc_\Gamma v=f\},
\end{equation}
with which it is a Hilbert space unitarily isomorphic to the quotient space $H^s(\R^n)/K_s$, where $K_s$ denotes the kernel of $\widetilde \Trc_\Gamma:H^s(\R^n)\to L^2(\Gamma,\cH^d)$. This defines $\bH^{t}(\Gamma,\cH^d)$ for $t>0$. For $t<0$ we set $\bH^{t}(\Gamma,\cH^d):= (\bH^{-t}(\Gamma,\cH^d))^*$. Identifying $L^2(\Gamma,\cH^d)$ with its dual and setting $\bH^0(\Gamma,\cH^d):= L^2(\Gamma,\cH^d)$, we have that  $\bH^{t}(\Gamma,\cH^d)\subset \bH^{t'}(\Gamma,\cH^d)$ for $t>t'$, with continuous and dense embedding. 
Further, where
\begin{equation} \label{eq:td}
t^d := 1-\frac{n-d}{2},
\end{equation}
$\bH^{t^d}(\Gamma,\cH^d)$ is the space $\bH(\Gamma,\mu)$ defined by \eqref{eq:Hdef}, \eqref{eq:Hsnorm} coincides with \eqref{eq:Hnorm}, and $\bH^{-t^d}(\Gamma,\cH^d)=\bH^*(\Gamma,\mu)$.

In \cite[\S3(b)]{caetano2023integral} it is noted that, in the case we are considering that $\Gamma$ is a $d$-set and $\mu=\cH^d$,  the definition of the operator $\bA:\bH^{-t^d}(\Gamma,\cH^d)\to \bH^{t^d}(\Gamma,\cH^d)$, defined by \eqref{eq:bAdef}, can be uniquely extended to the larger domain $\bH^{t-t^d}(\Gamma,\cH^d)$, for $-t^d<t<0$ (indeed also to $t=-t^d$ if $d=n$), with the following mapping properties.

\begin{proposition}\label{prop:cont}(\cite[Prop.~3.18]{caetano2023integral}) Suppose that $\Gamma$ is a $d$-set, for some  $d\in (n-2,n]$, and $\mu=\cH^d$. Then $\bA:\bH^{t-t^d}(\Gamma,\cH^d)\to \bH^{t+t^d}(\Gamma,\cH^d)$ and is continuous, for $|t|<t^d$, indeed for $|t|\leq 1$ when $d=n$.
\end{proposition}

\noindent
Building on the above,  we have also invertibility for a range of $t$ in the case that $\Gamma$ satisfies the following assumption that uses terminology introduced in \S\ref{sec:meas}.
\begin{assumption} \label{ass:ifs}
$\Gamma$ is the attractor of an IFS satisfying the OSC with $d:=\dim_H(\Gamma)\in (n-2,n]$ and either: (a) $\Gamma$ is disjoint; or (b) $d\in \N$  and $\Gamma$ is contained in a $d$-dimensional affine subspace of $\R^n$. 
\end{assumption}
\noindent Recall that if $\Gamma$ satisfies this assumption then $\Gamma$ is a $d$-set, with $d=\dim_H(\Gamma)$, and  note that condition (b) includes the case $d=n$ when $\Gamma$ is an $n$-set. 

The following proposition is a consequence of: Proposition \ref{prop:cont}; Corollary \ref{cor:iefinal} (which proves the claimed invertibility in the case $t=0$); a result on interpolation of invertibility of operators (\cite[Prop.~4.7]{MitTay99}, which quotes \cite{Snei74}); the results that, if Assumption \ref{ass:ifs} holds, then $\{\bH^t(\Gamma,\cH^d)\}_{|t|<1}$ is an interpolation scale for $d<n$ \cite[Cor.~3.3]{Caetano24}, while $\{\bH^t(\Gamma,\cH^d)\}_{t\geq 0}$  and $\{\bH^t(\Gamma,\cH^d)\}_{t\leq 0}$  are interpolation scales if $d=n$ \cite[\S1]{Paperdn}.  

\begin{proposition} \label{prop:invert}(\cite[Prop.~3.19]{caetano2023integral}) 
Suppose that $\Gamma$ satisfies \eqref{eq:GamRest} and Assumption \ref{ass:ifs}, $\mu=\cH^d$, and $k^2\not\in \SigO$. Then, for some $\varepsilon\in (0,t^{d'})$, where $d':= \dim_H(\partial O)$, $\bA:\bH^{t-t^d}(\Gamma,\cH^d)\to \bH^{t+t^d}(\Gamma,\cH^d)$ is invertible, for $|t|\leq\varepsilon$.
\end{proposition}

\noindent Under the conditions of the above proposition, $d'\leq d$, since $\partial O\subset \Gamma$. 
If $d'<d$, which is necessarily the case if $d=n$ by \cite[Prop.~2.5]{Paperdn}, then
 $\bA:\bH^{t-t^d}(\Gamma,\cH^d)\to \bH^{t+t^d}(\Gamma,\cH^d)$ is continuous for $t\in [-t^d,t^{d}]$, by Proposition \ref{prop:cont}, but cannot be invertible for any $t> t^{d'}$ (indeed any $t\geq t^{d'}$ if $\partial O$ is a $d'$-set) as discussed below Hypothesis 3.21 in \cite{caetano2023integral}.

\begin{remark} \label{rem:reg} Proposition \ref{prop:invert} implies that, if $\Gamma$ satisfies \eqref{eq:GamRest} and Assumption \ref{ass:ifs}, $\mu=\cH^d$, and $k^2\not\in \SigO$, then $\bA:\bH^{t-t^d}(\Gamma,\cH^d)\to \bH^{t+t^d}(\Gamma,\cH^d)$ is invertible for some $t\in (0,t^{d'})$. As noted in \cite[Prop.~3.19]{caetano2023integral} this implies that, if $\phi\in H_\Gamma^{-1}$ and $A\phi=g$, for some $g\in V_1(\Gamma^c)$, with $\mathfrak{g} := \widetilde \Trc_\Gamma g \in \bH^{t+t^d}(\Gamma,\cH^d)$, then $\phi\in H_\Gamma^{-1+t}$, with $\|\phi\|_{H^{-1+t}(\R^n)} \leq C\|\mathfrak{g}\|_{\bH^{t+t^d}(\Gamma,\cH^d)}$, for some $C>0$ independent of $\phi$ and $g$. Further, with $g$ given by \eqref{eq:iemain2}, since $u^{\mathrm{inc}}$ satisfies \eqref{eq:he} in a neighbourhood of $O$, 
we can choose $\chi^{\mathrm{inc}}\in C_{0,O}^\infty$ so that $\chi^{\mathrm{inc}} u^{\mathrm{inc}}\in C_0^\infty(\R^n)$. Thus $\mathfrak{g}=-\widetilde \Trc_\Gamma(\chi u^{\mathrm{inc}}) =-\widetilde \Trc_\Gamma(\chi^{\mathrm{inc}} u^{\mathrm{inc}}) = - u^{\mathrm{inc}}|_\Gamma \in \bH^{t+t^d}(\Gamma,\cH^d)$, for all $t>0$, with $\|\mathfrak{g}\|_{\bH^{t+t^d}(\Gamma,\cH^d)} \leq \|\chi^{\mathrm{inc}} u^{\mathrm{inc}}\|_{H^{1+t}(\R^n)}$.
\end{remark}

We will make use of the above results in the case that $\Gamma$ and $\mu$ are given by \eqref{eq:musum} and \eqref{eq:GamRest} and \eqref{eq:reg} hold. In such cases we may choose $\chi_j \in C^\infty_{0,\Gamma_j}$, for $j=1,\ldots, J$, such that  $\mathrm{supp}(\chi_i)\cap\mathrm{supp}(\chi_j)=\emptyset$, $i\neq j$, and $\mathrm{supp}(\chi_j)\subset \mathrm{supp}(\chi)$, $j=1,\ldots,J$. Note that, if $\psi\in H_\Gamma^{-1}$ and $\psi_j:=\chi_j\psi\in H^{-1}_\Gamma$, then
\begin{equation} \label{eq:summ}
\psi = \sum_{j=1}^J \psi_j.
\end{equation}
Further, for $j=1,\ldots,J$, let $\bA_j$ denote the operator $\bA$ in the case that $\Gamma=\Gamma_j$, let $\Omega_j$ denote the unbounded component of $\Gamma_j^c$ so that 
$\Omega\subset \Omega_j$, and let $O_j:= \Omega_j^c$, and $\Omega_{-,j} := O_j\setminus \Gamma_j$. For $j=1,\ldots,J$, note that $O_j$ is compact, and that $\partial O_j \subset \Gamma_j\subset O_j\subset O$ so that $\partial \Omega_{-,j} \subset \Gamma_j$.
Note that $\dim_H(\partial O_j) \geq n-1$ if $O_j$ has interior points (see \eqref{eq:boundH}), otherwise $O_j=\partial O_j=\Gamma_j$ so that $\dim_H(\partial \Omega_j)=d_j$. With these notations, in particular with $\chi_j$, $j=1,\ldots,J$, defined as above, we have the following result. See Remark \ref{rem:multiple} for discussion of the assumptions this theorem makes on $k^2$.

\begin{theorem} \label{thm:gen}
Suppose $\Gamma$ satisfies \eqref{eq:GamRest}, $\Gamma$ and $\mu$ are given by \eqref{eq:musum}, \eqref{eq:reg} holds, and $k^2\not\in \SigO$. Suppose further that $k^2\not\in \sigma(-\Delta_D(\Omega_{-,j}))$ and $\Gamma_j$ satisfies Assumption \ref{ass:ifs}, for $j=1,\ldots,J$. Then there exists $\varepsilon=(\varepsilon_1,\ldots,\varepsilon_J)\in \R^J$ with $\varepsilon_j\in (0,t^{d'_j})$, $j=1,\ldots,J$, where $d'_j:= \dim_H(\partial O_j)$ and $t^{d'_j} := 1-(n-d'_j)/2\in (0,1]$, such that $\bA_j:\bH^{t-t^{d_j}}(\Gamma_j,\cH^{d_j})\to \bH^{t+t^{d_j}}(\Gamma_j,\cH^{d_j})$ is invertible for $|t|\leq \varepsilon_j$ and $j=1,\ldots,J$. Further, if $m\in \{1,\ldots,J\}$, $t\in (0,t^{d'_m})$, $\bA_m:\bH^{t-t^{d_m}}(\Gamma_m,\cH^{d_m})\to \bH^{t+t^{d_m}}(\Gamma_m,\cH^{d_m})$ is invertible, $\phi\in H^{-1}_\Gamma$ is the unique solution of \eqref{eq:iemain2} (with $g$ defined as in \eqref{eq:iemain2}), $\phi_j:= \chi_j\phi$, for $j=1,\ldots,J$, and $\chi^{\mathrm{inc}}\in C_{0,O}^\infty$ is such that $\chi^{\mathrm{inc}}u^{\mathrm{inc}}\in C_0^\infty(\R^n)$, then 
$\phi_m\in H_{\Gamma_m}^{-1+t}$ with 
\begin{equation}\label{bound:gen}
\|\phi_m\|_{H^{-1+t}(\R^n)} \leq C\|\chi^{\mathrm{inc}} u^{\mathrm{inc}}\|_{H^{1+t}(\R^n)}, 
\end{equation}
 for some constant $C>0$ independent of $\phi$, $g$, and $u^{\mathrm{inc}}$.
\end{theorem}

\begin{remark}[Is the assumption in the above theorem that $k^2\not\in \sigma(-\Delta_D(\Omega_{-,j}))$ superfluous?] \label{rem:multiple}
If 
\begin{equation} \label{eq:OiOj}
O_i\cap O_j=\emptyset, \qquad i\neq j,
\end{equation}
 then $O_i\cap \Gamma_j=\emptyset$ for $i\neq j$. Thus, for $j=1,\ldots,J$, $\Omega_{-,j}=O_j\setminus \Gamma_j = O_j\setminus \Gamma\subset O\setminus \Gamma=\Omega_-$ and $\Omega_- \cap \partial \Omega_{-,j} \subset \Omega_- \cap \Gamma_j=\emptyset$, so that $\partial \Omega_{-,j} \subset \partial \Omega_-$ and $u\in H_0^1(\Omega_-)$ implies $u|_{\Omega_{-,j}}\in H_0^1(\Omega_{-,j})$.  
Thus, if \eqref{eq:OiOj} holds and $k^2\not \in \SigO$, then  $k^2\not\in \sigma(-\Delta_D(\Omega_{-,j}))$, for $j=1,\ldots,J$.

But $k^2\not \in \SigO$ does not imply that  $k^2\not\in \sigma(-\Delta_D(\Omega_{-,j}))$ in general. For example, with $n=2$ and $J=2$, let $\Omega=\{x\in \R^2:|x|>2\}$ and $\Gamma_j:= \partial B_j(0)$, for $j=1,2$, so that $\Gamma=\partial B_1(0)\cup \partial B_2(0)$, $O_j=B_j(0)$, and  $\Omega_j=O_j^c$. Then $\Omega_{-,j}= \{x\in \R^2:|x|<j\}$, for $j=1,2$, while $\Omega_-=\{x\in \R^2:|x|<1 \mbox{ or } 1<|x|<2\}$. Thus  $k^2\not \in \SigO$ does not imply that  $k^2\not\in \sigma(-\Delta_D(\Omega_{-,2}))$, in particular
 the lowest eigenvalue of $-\Delta_D(\Omega_{-,2})$ is not in $\SigO$.
\end{remark}

\begin{proof}[Proof of Theorem \ref{thm:gen}.]
The third sentence is an immediate consequence of Proposition \ref{prop:invert}. To see that the last sentence of the theorem holds suppose that $\bA_m:\mathbb{H}^{t-t_{d_m}}(\Gamma_m, \mathcal{H}^{d_m})\to \mathbb{H}^{t+t_{d_m}}(\Gamma_m, \mathcal{H}^{d_m})$ is invertible for some $m\in \{1,\ldots,J\}$ and $t\in (0, t^{d'_m})$. Define $g_\ell:=-P_\ell(\chi u^{\mathrm{inc}})$ for $\ell\in \{ 1,\ldots,J\}$, where $P_\ell$ is the orthogonal projection $P_\ell:H^1(\R^n)\to V_1(\Gamma_\ell^c)$.  For $j,\ell\in \{1,\ldots,J\}$ define $A_{\ell,j}:H^{-1}_{\Gamma_j}\to V_1(\Gamma_\ell^c)$ 
by $A_{\ell,j}\psi :=P_\ell(\chi\cA\psi)=P_\ell A\psi$, $\psi\in H_{\Gamma_j}^{-1}$. (Note that $A_{m,m}$ is just the operator $A$ in the case that $\Gamma=\Gamma_m$.) Since $\phi=\sum_{j=1}^J \phi_j$, equation \eqref{eq:iemain2}, i.e., $A\phi=g$, where $g:=-P(\chi u^{\mathrm{inc}})$, can be rewritten as
\begin{equation}
\sum_{j=1}^J A\phi_j=g.
\end{equation}
rearranging
we get
\begin{equation}\label{decomposedIE}
A_{m,m}\phi_m=
g_m-\sum_{\substack{j=1\\j\neq m}}^J A_{m,j}\phi_j. 
\end{equation}
By Remark \ref{rem:reg}, $\mathfrak{g}_m:=\widetilde{\mathrm{Tr}}_{\Gamma_m}g_m=\widetilde{\mathrm{Tr}}_{\Gamma_m}g=-u^{\mathrm{inc}}|_{\Gamma_m}\in \mathbb{H}^{\tau}(\Gamma_m, \mathcal{H}^{d_m})$ for all $\tau>0$ and the same is true of the second term in \eqref{decomposedIE} because $\widetilde{\mathrm{Tr}}_{\Gamma_m}A_{m,j}
\phi_j=\widetilde{\mathrm{Tr}}_{\Gamma_m}\chi_m\cA\phi_j$ and $\cA\phi_j\in C^\infty(\Gamma_j^c)$, so that $\chi_m\cA \phi_j\in C^\infty_0(\R^n)$ whenever $j\neq m$.
 Thus $\widetilde \Trc_\Gamma A_{m,m}\phi_m\in \bH^{\tau}(\Gamma_m, \cH^{d_m})$ for all $\tau>0$. By the invertibility of $\bA_m$, applying Remark \ref{rem:reg} with $\Gamma=\Gamma_m$, we get $\phi_m\in H^{-1+t}(\Gamma_m)$.

For every $\tau>0$, and some constant $C_{\tau}>0$ independent of $\phi$, we have for the terms of the sum in \eqref{decomposedIE} that
\begin{equation*}
\|\widetilde{\mathrm{Tr}}_{\Gamma_m}A_{m,j}\phi_j\|_{\mathbb{H}^\tau(\Gamma_m,\mathcal{H}^{d_m})}\leq \|\chi_m\cA \phi_j\|_{H^{\tau+(n-d_m)/2}(\R^n)}\leq C_{\tau} \|\phi_j\|_{H^{-1}(\R^n)}. 
\end{equation*}
Since $\phi_j=\chi_j \phi$ and $\phi=A^{-1}g$, and $g=-P(\chi u^{\mathrm{inc}})=-P(\chi^{\mathrm{inc}} u^{\mathrm{inc}})$, it follows further that there is some constant $C'_{\tau}>0$, independent of $\phi$, $g$, and $u^{\mathrm{inc}}$, such that
\begin{equation}\label{svash1}
\|\widetilde{\mathrm{Tr}}_{\Gamma_m}(\chi_m\cA\phi_j)\|_{\mathbb{H}^\tau(\Gamma_m,\mathcal{H}^{d_m})}\leq C'_{\tau} \lVert g\rVert_{H^1(\R^n)}\leq C'_{\tau} \lVert \chi^{\mathrm{inc}} u^{\mathrm{inc}}\rVert_{H^1(\R^n)}.
\end{equation}
We have a similar bound for the remaining term on the right-hand side of \eqref{decomposedIE}, recalling from Remark \ref{rem:reg} that 
\begin{equation}\label{svash2}
\|\widetilde{\mathrm{Tr}}_{\Gamma_m}g_m\|_{\mathbb{H}^\tau(\Gamma_m,\mathcal{H}^{d_m})} =\|\widetilde{\mathrm{Tr}}_{\Gamma_m}(\chi^{\mathrm{inc}} u^{\mathrm{inc}})\|_{\mathbb{H}^\tau(\Gamma_m,\mathcal{H}^{d_m})}\leq \| \chi^{\mathrm{inc}} u^{\mathrm{inc}}\|_{H^{\tau+(n-d_m)/2}(\R^n)}. 
\end{equation}
Finally, to obtain the bound \eqref{bound:gen} on $\phi_m$, we apply the estimates \eqref{svash1} and \eqref{svash2} with $\tau=t+t^{d_m}$ so that $\tau+(n-d_m)/2=1+t$, and apply the bound in Remark \ref{rem:reg}, in the case that $\Gamma=\Gamma_m$ so that $A=A_{m,m}$. We obtain, where $C>0$ is the constant in that bound, that
\[
\| \phi_m\|_{H^{-1+t}(\R^n)}\leq C \bigg( \| \chi^{\mathrm{inc}} u^{\mathrm{inc}}\|_{H^{1+t}(\R^n)}+(J-1) C'_{t+t^{d_m}}\| \chi^{\mathrm{inc}} u^{\mathrm{inc}}\|_{H^1(\R^n)}\bigg),
\]
which implies the claimed bound since $\| \chi^{\mathrm{inc}} u^{\mathrm{inc}}\|_{H^1(\R^n)} \leq \| \chi^{\mathrm{inc}} u^{\mathrm{inc}}\|_{H^{1+t}(\R^n)}$.
\end{proof}

\subsection{Convergence rates} \label{sec:conv}

In this section we establish convergences rates for the piecewise-constant Galerkin method introduced in \S\ref{sec:GM1} under additional constraints on $\Gamma$ and $\mu$. 
We will require throughout this section that $\Gamma$ and $\mu$ are given by \eqref{eq:musum} and that the following assumption holds, for $i,j=1,\ldots,J$: 
\begin{equation} \label{eq:Gj}
\begin{array}{lll}
\mbox{(i)} & \mbox{$\Gamma_j$ is the attractor of an IFS satisfying the OSC with $d_j= \dim_H(\Gamma_j)\in (n-2,n]$;}\\
\mbox{(ii)} & \mbox{for $i\neq j$, $\mu(\Gamma_i\cap \Gamma_j)=0$  if $d_j\neq n$ and $d_i\neq n$, while $\Gamma_i\cap \intt(\Gamma_j)=\emptyset$ if $d_j=n$.}
\end{array}
\end{equation}
 Under these assumptions Corollary \ref{cor:GM2} applies so that to establish rates of convergence  for $\|\phi-\phi_N\|_{H^{-1}(\R^n)}$ and $|u(x)-u_N(x)|$, for $x\in \Omega$,  it remains to estimate the best approximation errors $e_N(\phi)$ and $e_N(\zeta_x)$. 
 
 \begin{example} \label{ex:Koch}
 As an example of a $\Gamma$ and $\mu$ satisfying these assumptions, suppose that $n=2$ and $\Gamma_1:= O$ is a solid Koch snowflake, which is the attractor of an IFS satisfying the OSC with $d_1=\dim_H(\Gamma_1)=n$, as recalled in \cite[Ex.~2.4]{caetano2023integral}. Then $\partial O = \Gamma_2 \cup \Gamma_3\cup \Gamma_4$ is the union of three Koch curves, each the attractor of an IFS satisfying the OSC with $d_j:= d = \log(4)/\log(3)$, $j=2,3,4$ \cite[Ex.~2.3, \S5.1.3]{caetano2023integral}. Thus $\Gamma$ and $\mu$, given by \eqref{eq:musum} with $J=4$ so that $\Gamma=O$, satisfy the above assumptions. In particular, for $i,j=2,3,4$, $\mu(\Gamma_i\cap \Gamma_j)=0$ since $\Gamma_i$ and $\Gamma_j$ intersect only at a single point. Note that, if we choose $a_2=a_3=a_4$ in \eqref{eq:musum}, $\mu$ takes the form \eqref{eq:muo}, with $F=\overline{\intt(O)}=O$.
 \end{example}   
 
 In the case that $\Gamma$ and $\mu$ satisfy the conditions \eqref{eq:musum} and \eqref{eq:Gj}  we will consider {\em self-similar meshes} on $\Gamma$ constructed in the following manner. Suppose first that $J=1$ in \eqref{eq:musum}, so that $\Gamma$ is the attractor of an IFS consisting of a set $S=\{s_1,\ldots,s_M\}$ of contracting similarities that satisfy the OSC, so that \eqref{eq:ss} holds. Let $\rho_{\min}$ denote the smallest of the contraction factors. Then, given $h\leq \mathrm{diam}(\Gamma)$, as discussed above Theorem 4.5 in \cite{caetano2023integral}, there is a unique mesh $\tilde \tau_h(\Gamma)$ on $\Gamma$ such that: (i) \eqref{eq:mesh} holds; and (ii) for each element $T\in \tilde \tau_h(\Gamma)$, $\rho_{\min} h <\mathrm{diam}(T)\leq h$ and $T$ is the image of $\Gamma$ under some composition of elements of $S$, i.e., for some $\ell\in \N_0$, $T=\sigma_1\circ\sigma_2\ldots \circ \sigma_\ell(\Gamma)$, with $\sigma_j\in S$, $j=1,\ldots,\ell$. If $h=\mathrm{diam}(\Gamma)$, then $\tilde \tau_h(\Gamma)=\{\Gamma\}$. If $h<\mathrm{diam}(\Gamma)$, then $\tilde \tau_h(\Gamma)$ is a mesh consisting of finitely many smaller copies of $\Gamma$. 

In this case that $J=1$ the self-similar meshes that we will focus on are the meshes $\tau_h(\Gamma)$, for $h\leq \mathrm{diam}(\Gamma)$, given by 
\begin{equation} \label{eq:tauh}
\tau_h(\Gamma) := \left\{\begin{array}{ll} \tilde \tau_h(\Gamma), & \mbox{if } d<n,\\
\{\intt(T): T\in \tilde \tau_h(\Gamma)\}, & \mbox{if } d=n. \end{array}\right. 
\end{equation}
In the case $d=n$, since each $T\in \tilde \tau_h(\Gamma)$ is congruent to $\Gamma$, if $T\in \tilde \tau_h(\Gamma)$ then $T=\overline{\intt(T)}$ (in particular $\intt(T)\neq \emptyset$) and $\mu(\partial T)=|\partial T|=0$ (see, e.g., \cite[Prop.~2.5]{Paperdn}). Thus $\tau_h(\Gamma)$ satisfies \eqref{eq:mesh}, since $\tilde \tau_h(\Gamma)$ does.
Moreover, for all $d\in(n-2,n]$, since $\tau_h(\Gamma)$ satisfies \eqref{eq:mesh} and $\mu$ is a multiple of $\cH^d$,
$$
\mu(\Gamma) = \sum_{T\in \tau_h(\Gamma)}\mu(T) \quad \mbox{and} \quad (\rho_{\min}h/\mathrm{diam}(\Gamma))^d \mu(\Gamma) \leq  \mu(T) \leq (h/\mathrm{diam}(\Gamma))^d\mu(\Gamma), \quad \mbox{for }T\in \tau_h(\Gamma),
$$
by \cite[Eqn.~(2.5)]{Falconer2014}. Thus the number of elements, $\#\tau_h(\Gamma)$, of $\tau_h(\Gamma)$ satisfies
\begin{equation} \label{eq:dof}
(h/\mathrm{diam}(\Gamma))^{-d} \leq \#\tau_h(\Gamma)\ \leq (\rho_{\min}h/\mathrm{diam}(\Gamma))^{-d}.
\end{equation}

In the case that $J>1$, $\Gamma = \Gamma_1\cup\ldots\cup \Gamma_J$, where each $\Gamma_j$ is the attractor of an IFS satisfying the OSC and $\mu$ is given by \eqref{eq:musum}. Given a vector $h=(h_1,\ldots, h_J)$, with each $h_j\in (0, \mathrm{diam}(\Gamma_j))$,  the meshes on  $\Gamma$ that we will consider are the meshes
\begin{equation} \label{eq:meshIFS}
T_h(\Gamma):= \bigcup_{j=1}^J \tau_{h_j}(\Gamma_j).
\end{equation}  
Note that these meshes satisfy \eqref{eq:mesh} because each mesh $\tau_{h_j}$ is a mesh of $\Gamma_j$ in the sense of \eqref{eq:mesh} for the measure $\cH^{d_j}|_{\Gamma_j}$ and because of our requirements  \eqref{eq:Gj} and \eqref{eq:tauh}; note that $\tau_{h_j}(\Gamma_j)\subset \intt(\Gamma_j)$ if $d_j=n$. Note also that, if $O$, $\Gamma$ and $\mu$ are as given as in Example \ref{ex:Koch}, then the mesh $T_h(\Gamma)$
is a particular instance of the mesh $\tau_h$ defined in \S\ref{sec:onset} above \eqref{eq:Vh}, so that Corollary \ref{cor:GM3} applies with $\tau_h=T_h(\Gamma)$. 

\begin{remark}[Numerical integration and implementation] \label{rem:ni}
Given the assumptions \eqref{eq:musum} and \eqref{eq:Gj} that we make on $\Gamma$ and $\mu$, a key attraction of these self-similar meshes is that, to evaluate the regular and singular integrals that appear in Proposition \ref{prop:coeff}, as needed to implement the Galerkin method, there is potential to use recently proposed quadrature rules that are adapted to these measures and these types of meshes \cite{Gibbsetal23,Gibbsetal24,Jolyetal24}. In the case that $J=1$, a fully discrete implementation of the Galerkin method with these meshes is described and implemented in \cite{caetano2023integral} (and see \cite[\S5.4]{caetano2024integral}). Further, for any $J\in \N$, with the mesh \eqref{eq:meshIFS} and the assumptions we make on $\Gamma$ and $\mu$, the integrals in \eqref{eq:uNrep} (which have smooth integrands) can be evaluated using either the barycentre rule of \cite{Gibbsetal23} or the high order rules of \cite{Jolyetal24}. We leave further discussion of implementation details for the case $J>1$ to a future paper.
\end{remark}

A further attraction of these self-similar meshes  is that, under certain additional assumptions on the components $\Gamma_1,\ldots,\Gamma_J$ of $\Gamma$, we can derive the following best approximation estimate, as an immediate corollary of \cite[Thm.~4.5]{caetano2023integral}. Note that the assumptions of this proposition imply that also \eqref{eq:Gj} holds.
\begin{proposition} \label{prop:convrates}
Suppose $\Gamma$ and $\mu$ are given by \eqref{eq:musum}, $\Gamma_j$ satisfies Assumption \ref{ass:ifs}, for $j=1,\ldots,J$, and \eqref{eq:reg} holds. Choose $\chi_1,\ldots,\chi_J$ as above \eqref{eq:summ}, and suppose that $\psi\in H_\Gamma^{-1}$ is such that $\psi_j := \chi_j\psi\in H^{s_j}_{\Gamma_j}$, for $j=1,\ldots,J$, where $-1<s_j<-(n-d_j)/2$. Then there exists $C>0$ such that,
given any sequence $h_N=(h_{1,N},\ldots,h_{J,N})$, with $0<h_{j,N}\leq \mathrm{diam}(\Gamma_j)$ for $1\leq j\leq J$ and $N\in \N$, it holds that
$$
e_N(\psi) = \min_{\psi_N\in V_N}\|\psi-\psi_N\|_{H^{-1}(\R^n)} \leq C\sum_{j=1}^J h_{j,N}^{1+s_j} \|\psi_j\|_{H^{s_j}(\R^n)},
$$
where $V_N := \widetilde \Trc^*_\Gamma(\widetilde V_N)$, $\widetilde V_N$ is given by \eqref{eq:tilVN}, and the mesh $\tau_N=\{T_{\ell,N}\}_{\ell=1}^{M_N}$ is given by $\tau_N:= T_{h_N}(\Gamma)$.
\end{proposition}
\begin{proof}
 For $j=1,\ldots,J$, define $V_{j,N}\subset V_N$ by   $V_{j,N} := \widetilde \Trc^*_\Gamma(\widetilde V_{j,N})$, where $\widetilde V_{j,N}:= \{f\in \widetilde V_N:\supp(f)\subset \Gamma_j\}$. By \cite[Thm.~4.5]{caetano2023integral},
 $$
\min_{\psi_{j,N}\in V_{j,N}}\|\psi_j-\psi_{j,N}\|_{H^{-1}(\R^n)} \leq Ch_{j,N}^{1+s_j} \|\psi_j\|_{H^{s_j}(\R^n)}, \qquad j=1,\ldots,J,
$$
from which the result follows by \eqref{eq:summ}.
 \end{proof}

Combining this best approximation error estimate with  Corollary \ref{cor:GM2} and Theorem \ref{thm:gen}, we obtain the following result on the rate of convergence of the Galerkin method, in which we use the notations defined in and above Theorem \ref{thm:gen}.

\begin{theorem} \label{thm:finalconv}
Suppose that $\Gamma$ satisfies \eqref{eq:GamRest}, $\Gamma$ and $\mu$ are as in Theorem \ref{thm:summ}, $\Gamma_j$ satisfies Assumption \ref{ass:ifs}, for $j=1,\ldots,J$, and \eqref{eq:reg} holds. Suppose further that $k^2\not\in \SigO$, $k^2\not\in \sigma(-\Delta_D(\Omega_{-,j}))$, for $j=1,\ldots,J$, $x\in \Omega$, and $\chi^{\mathrm{inc}}, \chi_x\in C^\infty_{0,O}$ are such that $u^{\mathrm{inc}}$ satisfies \eqref{eq:he} in a neighbourhood of $\supp(\chi^{\mathrm{inc}})$ and $x\not \in \supp(\chi_x)$. 
 Then there exist $h>0$, $C>0$, and  $t=(t_1,\ldots,t_J)\in \R^J$, with $t_j\in (0,t^{d'_j})$, $j=1,\ldots,J$,  each independent of $u^{\mathrm{inc}}$ and $x$, such that, given any sequence $h_N=(h_{1,N},\ldots,h_{J,N})$, with $0<h_{j,N}\leq \max(h,\mathrm{diam}(\Gamma_j))$ for $1\leq j\leq J$ and $N\in \N$, if the mesh $\tau_N=\{T_{\ell,N}\}_{\ell=1}^{M_N}$ is given by $\tau_N:= T_{h_N}(\Gamma)$, and $V_N := \widetilde \Trc^*_\Gamma(\widetilde V_N)$, where $\widetilde V_N$ is given by \eqref{eq:tilVN}, then the Galerkin method equations \eqref{eq:GM} and \eqref{eq:GM2} have exactly one solution and 
\begin{eqnarray*} 
\|\phi-\phi_N\|_{H^{-1}(\R^n)} &\leq& C\sum_{j=1}^J h_{j,N}^{t_j} \, \|\chi^{\mathrm{inc}} u^{\mathrm{inc}}\|_{H^{2}(\R^n)},\\ 
|u(x)-u_N(x)| &\leq& C\left(\sum_{j=1}^J h_{j,N}^{t_j}\right)^2 \, \|\chi^{\mathrm{inc}} u^{\mathrm{inc}}\|_{H^{2}(\R^n)}\|\chi_x \Phi(x,\cdot)\|_{H^{2}(\R^n)},
\end{eqnarray*}
where $\phi_N$ and $u_N$ are given by \eqref{eq:phiN} and \eqref{eq:uNrep}.
\end{theorem}
\begin{proof} 
Suppose that  the assumptions in the first two sentences of the theorem are satisfied. Then, by Corollary \ref{cor:GM2},  there exist $h>0$ and $C>0$, independent of $u^{\mathrm{inc}}$ and $x$, such that, for any sequence $h_N=(h_{1,N},\ldots,h_{J,N})$, mesh $\tau_N$, and approximation space $V_N$ satisfying the conditions of the theorem,  the Galerkin method equations \eqref{eq:GM} and \eqref{eq:GM2} have exactly one solution and  $\|\phi-\phi_N\|_{H^{-1}(\R^n)} \leq C\, e_N(\phi)$,  $|u(x)-u_N(x)| \leq C\, e_N(\phi)e_N(\zeta_x)$, where $\zeta_x$ is as in Proposition \ref{prop:GM}.  By Theorem \ref{thm:gen} there exists $t=(t_1,\ldots,t_J)\in \R^J$ with $t_j\in (0,t^{d'_j})$, $j=1,\ldots,J$,  such that $\bA_j:\bH^{t_j-t^{d_j}}(\Gamma_j,\cH^{d_j})\to \bH^{t_j+t^{d_j}}(\Gamma_j,\cH^{d_j})$ is invertible for  $j=1,\ldots,J$. This implies, also by Theorem \ref{thm:gen}, that there exists $C'>0$, independent of $u^{\mathrm{inc}}$ and $x$, such that, for $j=1,\ldots,J$,
$$
\|\phi_j\|_{H^{-1+t_j}(\R^n)} \leq C'\|\chi^{\mathrm{inc}} u^{\mathrm{inc}}\|_{H^{1+t_j}(\R^n)},
$$
where $\phi_j := \chi_j \phi$ and $\chi_1,\ldots,\chi_J$ are as defined as above \eqref{eq:summ}. By the same argument we have also that
$$
\|\zeta_{x,j}\|_{H^{-1+t_j}(\R^n)} \leq C'\|\chi_x \Phi(x,\cdot)\|_{H^{1+t_j}(\R^n)},
$$
for $j=1,\ldots,J$, where $\zeta_{x,j} := \chi_j \zeta_x$, and note that 
$$\|\chi^{\mathrm{inc}} u^{\mathrm{inc}}\|_{H^{1+t_j}(\R^n)}\leq \|\chi^{\mathrm{inc}} u^{\mathrm{inc}}\|_{H^{2}(\R^n)}, \qquad \|\chi_x \Phi(x,\cdot)\|_{H^{1+t_j}(\R^n)}\leq \|\chi_x \Phi(x,\cdot)\|_{H^{2}(\R^n)},
$$ 
since $t_j<t^{d'_j}\leq 1$. It follows by Proposition \ref{prop:convrates} that, for some constant $C''>0$ independent of $u^{\mathrm{inc}}$ and $x$,
$$
e_N(\phi) \leq C''\sum_{j=1}^J h_{j,N}^{t_j} \, \|\chi^{\mathrm{inc}} u^{\mathrm{inc}}\|_{H^{2}(\R^n)}, \qquad e_N(\zeta_x) \leq C''\sum_{j=1}^J h_{j,N}^{t_j} \, \|\chi_x \Phi(x,\cdot)\|_{H^{2}(\R^n)},
$$ 
and the result follows.
\end{proof}

\subsubsection*{Acknowledgements}
\addcontentsline{toc}{section}{Acknowledgements}
We are grateful to Michael Hinz (Bielefeld) for stimulating discussions in relation to this work. The authors were supported in part by the CNRS INSMI IEA (International Emerging Actions 2022) grant
``Functional and applied analysis with fractal or non-Lipschitz boundaries”. The first and third authors thank the Isaac Newton Institute for Mathematical Sciences, Cambridge, for support and hospitality during the programme ``Geometric spectral theory and applications'', where work on this paper was undertaken. This work was supported by EPSRC grant EP/Z000580/1. The first and fifth
authors acknowledge, respectively, the support of the New Zealand Marsden Fund
(Grant No.~MFP-UOA2527) and the  support of an  EPSRC PhD Studentship.

\appendix
\section{Appendix} \label{sec:app}

\subsection{Measures, $d$-sets, Hausdorff dimension, IFS attractors} \label{sec:meas}

The measures $\mu$ on $\R^n$ that concern us are largely Radon measures, i.e.\   (positive) measures that are  regular Borel measures, as defined, e.g., in \cite{Rudin}, with $\mu(F)<\infty$ for every compact $F\subset \R^n$. Given a closed set $F\subset \R^n$ and $0\leq d\leq n$ we call a Borel measure $\mu$ with $\supp(\mu) \subset F$ an {\em upper $d$-regular} measure on $F$ if, for some $c>0$,
\begin{equation} \label{eq:udr}
\mu(B_r(x)) \leq c r^d, \qquad x\in F, \;\; 0<r\leq 1,
\end{equation}
and call $\mu$  a {\em lower $d$-regular} measure on $F$ if the same inequality holds but with the $\leq$ replaced by $\geq$.  We call $\mu$ {\em $d$-regular} on $F$ if $\mu$ is both upper and lower $d$-regular\footnote{Some authors use {\em Ahlfors $d$-regular} or {\em Ahlfors-David $d$-regular} in place of $d$-regular, e.g., \cite{Biegart:09}, \cite[p.~92]{Mattila95}.}. We note that if $\mu$ is upper $d$-regular on a closed set $F$, then it is also upper $d$-regular on every closed set $\widetilde F\supset F$, in particular\footnote{\label{foot:8} If \eqref{eq:udr} holds then the same inequality holds, with $c$ replaced by some $c'\geq c$, for all $x\in \R^n$. For if $x\in \R^n$ and $B_r(x)$ does not intersect $F$ then $\mu(B_r(x))=0$, while if it intersects $F$ at some $y\in F$, then $B_r(x)\subset B_{2r}(y)$ so that $\mu(B_r(x))\subset \mu(B_{2r}(y))\leq 2^d cr^d$, for $0<r\leq 1/2$. It follows that, for $1/2<r\leq 1$ and $x\in \R^n$, $\mu(B_r(x))\leq \mu(B_1(x))\leq 4^d cc_n$, where $c_n$ is the number of closed balls of radius $1/2$ needed to cover a closed ball of radius 1.} it is upper $d$-regular on $\R^n$.  As a consequence, if, for $j=1,2$, $\mu_j$ is a Borel measure that is upper $d_j$-regular, for some $d_j\in [0,n]$, on the closed set $F_j\subset \R^n$, then $\mu_1+\mu_2$ is upper $d$-regular, with $d=\min(d_1,d_2)$, on $F_1\cup F_2$.  

The following lemma, which is one half of \cite[Lem.~2.13]{Carvalho12} with the constant made more explicit (cf.\ \cite[Rem.~2.2]{Caetano24}), enables us to establish the existence of certain singular integrals.

\begin{lemma} \label{lem:ub} Suppose that, for some $d\in (0,n]$ and $D>0$, $\mu$ is an upper $d$-regular measure on a compact set $F\subset \R^n$ with diameter $\mathrm{diam}(F)\leq D$, so that \eqref{eq:udr} holds for some $c>0$, and suppose that $\tau:(0,\infty)\to [0,\infty)$ is continuous and non-increasing. Then, for some constant $c'>0$ depending only on $d$, $n$, $c$, and $D$,
$$
\int_{F} \tau(|x-y|)\rd \mu(y) \leq c'\int_0^{\mathrm{diam}(F)} r^{d-1}\tau(r) \rd r, \qquad x\in F.
$$
\end{lemma}

For $0\leq d\leq n$ let $\cH^d$ denote the Borel measure that is  Hausdorff $d$-measure
on $\R^n$, and let $\dimH(S)\in [0,n]$ denote the Hausdorff dimension of $S\subset\R^n$ (see, e.g.,\ \cite{Falconer2014}). 
For convenience 
we adopt 
the normalisation of \cite[Def.~2.1]{EvansGariepy}, so that $\cH^d$ coincides with Lebesgue measure for $d=n$. We recall that (e.g., \cite[p.~57]{Mattila95}), if $G\subset \R^n$ is $\cH^d$-measurable with $\cH^d(G)<\infty$, then $\cH^d|_G$ is a Radon measure. 

As in \cite[\S1.1]{JW84} and \cite[\S3]{Triebel97FracSpec}, given $0\leq d\leq n$, we say a closed set $F\subset \R^n$ is a {\em$d$-set} if there exists a Borel measure $\mu$ with $\supp(\mu)=F$ that is $d$-regular on $F$. Equivalently \cite[Thm.~3.4]{Triebel97FracSpec}, $F$ is a $d$-set if there exist $c_{2}>c_1>0$ such that
\begin{align}
\label{eq:dset}
c_{1}r^{d}\leq\mathcal{H}^{d}\big(F\cap B_{r}(x)\big)\leq c_{2}r^{d},\qquad x\in F,\quad0<r\leq1.
\end{align}
We note that if $F$ is a $d$-set then  \cite[Cor.~3.6]{Triebel97FracSpec}  $\cH^d|_F$ is a Radon measure and $\dimH(F)=d$, and that $d$-sets are also termed Ahlfors $d$-regular or Ahlfors-David $d$-regular sets, e.g., \cite[p~92]{Mattila95}. Note further that if $F$ is a closed set and $\mu$ is a Radon measure with $\supp(\mu)=F$ that is upper $d$-regular, then $\dimH(F)\geq d$ by Frostman's lemma (e.g., \cite[Thm.~8.8]{Mattila95}).

By an \textit{iterated function system of contracting similarities} (we abbreviate this whole phrase as  \textit{IFS})\footnote{A useful introduction to IFSs is \cite[Chap.~9]{Falconer2014}; the website \cite{RiddleWebSite} gives many examples of IFSs and their attractors.} we mean a collection
$\{s_1,s_2,\ldots,s_M\}$, for some $M\geq 2$,
where,
for each $m=1,\ldots,M$, $s_m:\R^{n}\to\R^{n}$, with
$|s_m(x)-s_m(y)| = \rho_m|x-y|$, $x,y\in \R^{n}$,
for some {\em contraction factor} $\rho_m\in (0,1)$.
The attractor of the IFS is the unique non-empty compact set $\Gamma$ satisfying
\begin{equation} \label{eq:fixedfirst}
\Gamma = s(\Gamma), \quad \mbox{where} \quad s(E) := \bigcup_{m=1}^M  s_m(E), \quad  E\subset \R^{n}.
\end{equation}
We restrict attention to IFSs that satisfy the standard \textit{open set condition (OSC)} \cite[Eqn.~(9.11)]{Falconer2014},
which
implies that the attractor $\Gamma$ is a $d$-set (see, e.g.,\ \cite[Thm.~4.7]{Triebel97FracSpec}),
where $d\in (0,n]$ is the unique solution of
$\sum_{m=1}^M  (\rho_m)^d = 1$. For a \textit{homogeneous }IFS, where $\rho_m=\rho \in (0,1)$ for $m=1,\ldots,M $, we have
$d = \log(M)/\log(1/\rho)$.
Returning to the general, not necessarily homogeneous case, the OSC
also implies (again, see \cite[Thm.~4.7]{Triebel97FracSpec}) that $\Gamma$ is \textit{self-similar}, meaning that the sets
\begin{align}
\label{eq:GammamDef}
\Gamma_m:=s_m(\Gamma),\qquad m=1,\ldots,M,
\end{align}
satisfy
\begin{equation} \label{eq:ss}
\cH^d(\Gamma_{m}\cap \Gamma_{m'})=0,  \qquad m\neq m',
\end{equation}
so that $\Gamma$ is decomposed by \eqref{eq:fixedfirst} into $M$ similar %
copies of itself whose pairwise intersections have zero measure.
In some of our results we make the additional assumption that the sets $\Gamma_1,\ldots,\Gamma_M $ are disjoint. If this holds we say that the IFS attractor $\Gamma$ is \emph{disjoint}, the OSC is automatically satisfied (e.g., \cite[Lem.~2.5]{Caetano24}), and $d<n$ (e.g.\ \cite[Lemma 2.6]{Caetano24}).

\subsection{Function spaces} \label{sec:fs}
We give brief details of the mainly standard function space notations that we need to state and prove our results.  To suit our acoustics applications, our function spaces are spaces of complex-valued functions.

\subsubsection{Sobolev spaces} As usual (e.g., \cite{mclean2000strongly}), $\cS(\R^n)$ denotes the Schwartz space of smooth, rapidly decreasing functions and  $\cS^*(\R^n)$ its dual space\footnote{Throughout, to suit a complex Hilbert space setting, our distributions and other functionals will be anti-linear rather than linear, so that our dual spaces are spaces of anti-linear continuous functionals.}, the space of tempered distributions. For $s\in \R$, $H^s(\R^n)$ denotes the Sobolev space of those $u\in \cS^*(\R^n)$ whose distributional Fourier transform\footnote{We choose the normalising constant in our Fourier transform definition so that the mapping $u\mapsto \hat u$ is unitary on $L^2(\R^n)$; precisely, $\hat u(\xi) := (2\pi)^{-n/2}\int_{\R^n}\re^{-\ri \xi\cdot x}u(x)\, dx$, $\xi\in \R^n$.} $\hat u$ is locally integrable and satisfies
\begin{equation} \label{eq:unorm}
\|u\|_{H^s(\R^n)} := \left(\int_{\R^n}|\hat u(\xi)|^2(1+|\xi|^2)^s\, \rd \xi\right)^{1/2} < \infty.
\end{equation}
As usual, we identify $L^2(\R)$ with $H^0(\R^n)$, so that $H^s(\R^n)\subset L^2(\R^n)$, for $s\geq 0$, and identify $H^{-s}(\R^n)$ with the dual space $(H^s(\R^n))^*$ through the duality pairing
\begin{equation} \label{eq:dual}
\langle u,v\rangle := \int_{\R^n} \hat u(\xi) \overline{\hat v(\xi)}\, d\xi, \qquad u\in H^{-s}(\R^n), \quad v\in H^s(\R^n),
\end{equation}
that extends the inner product on $L^2(\R^n)$. 

Given a domain (a non-empty open set) $\Omega\subset\mathbb{R}^n$, $H^1(\Omega):=\{u\in L^2(\Omega): \nabla u\in (L^2(\Omega))^n\}$, where $\nabla u$ is the weak gradient.
In the case $\Omega=\R^n$, this definition agrees with our definition above for $H^1(\R^n)$.  Let
$C_0^\infty(\Omega)\subset C^\infty(\R^n)$ denote the set of smooth functions which are compactly supported in $\Omega$ and let
 $
 H_0^1(\Omega) := {\overline{C^{\infty}_0(\Omega)}}^{H^1(\Omega)}.
 $
 
 We use  certain spaces of compactly supported and locally integrable functions. For $1\leq p<\infty$ let $L^p_{\mathrm{comp}}(\Omega)$ denote the space of those functions in $L^p(\Omega)$ that are compactly supported. A sequence $(u_n)$ is convergent in $L^p_{\mathrm{comp}}(\Omega)$ if it is convergent in $L^{p}(\Omega)$ and, for some bounded $V\subset \Omega$, $\supp(u_n)\subset V$ for each $n$. Let $\Lloc(\Omega)$ denote the set of functions $u$ on $\Omega$ that satisfy $u|_V\in L^2(V)$ for every bounded\footnote{Other authors use $\Lloc(\Omega)$ to denote the set of functions $u$ that satisfy $u|_V\in L^2(V)$ for every {\em compact} $V\subset \Omega$. This puts no constraint on $u$ near $\partial \Omega$, in contrast to our definition which has $\Lloc(\Omega)=L^2(\Omega)$ if $\Omega$ is bounded.} $V\subset \Omega$. A sequence $(u_n)\subset \Lloc(\Omega)$ is convergent if $(u_n|_V)$ is convergent in $L^2(V)$ for every bounded $V\subset \Omega$.  Let 
 \begin{equation} \label{eq:H1locOmega}
 H^{1,\text{loc}}(\Omega):= \{u\in \Lloc(\Omega):\nabla u\in (\Lloc(\Omega))^n\}=\{u\in \Lloc(\Omega): \chi|_\Omega \,u\in H^1(\Omega), \forall \chi\in C^\infty_0(\R^n)\}
\end{equation} 
and 
\begin{equation} \label{eq:H10locOmega}
H_0^{1,\text{loc}}(\Omega):=\{u\in H^{1,\text{loc}}(\Omega): \chi|_\Omega\, u\in H_0^1(\Omega), \forall \chi\in C^\infty_0(\mathbb{R}^n)\}.
\end{equation}
  
\subsubsection{Capacity and quasi-continuous representatives} \label{sec:cap}

We denote the ($H^1(\R^n)$-)capacity\footnote{See \cite[Rem.~3.2]{HeMo:16} for an account of alternative notations and of other, equivalent definitions for $\mathrm{cap}(E)$.} of a set $E\subset \R^n$ by $\mathrm{cap}(E)$. Following, e.g., \cite[Defns.~2.2.1-2.2.4]{adams2012function}, for compact $E$ we set
\begin{equation} \label{eq:capdef}
\mathrm{cap}(E) :=\inf\{\|u\|_{H^1(\R^n)}^2:u\in C_0^\infty(\R^n) \mbox{ and } u\geq 1 \mbox{ on $E$}\};
\end{equation}
for open $\Omega\subset \R^n$ we define $\mathrm{cap}(\Omega) := \sup_E\mathrm{cap}(E)$, where the supremum is taken over all compact $E\subset \Omega$; for general $E\subset \R^n$ (this consistent with the definition for $E$ compact by \cite[Prop.~2.2.3]{adams2012function}) we set  $\mathrm{cap}(E) := \inf_\Omega\mathrm{cap}(\Omega)$, where the infimum  is taken over all open $\Omega\supset E$.  

A property which holds outside a set of zero capacity is said to hold {\em quasi-everywhere (q.e.)}. A (complex-valued) function $u$ defined q.e.\ on $\R^n$ is called {\em quasi-continuous} if, for any $\varepsilon>0$, there is an open set $\Omega\subset \R^n$ with $\mathrm{cap}(\Omega)<\varepsilon$ such that $u$ is continuous on $\Omega^c$. Importantly, if $u$ is quasi-continuous, then \cite[Thm.~6.1.4]{adams2012function}
\begin{equation} \label{eq:aeqe}
u=0 \mbox{ a.e. (with respect to $n$-dimensional Lebesgue measure)} \; \Rightarrow \; u=0 \mbox{ q.e.}
\end{equation}

The elements of $H^1(\R^n)$ are equivalence classes of functions that are equal almost everywhere. Every $u\in H^1(\R^n)$ has a representative $\tilde u$ (i.e., an element of the equivalence class) that is quasi-continuous. Note that (e.g., \cite[Thm.~6.2.1]{adams2012function}), for each $u\in H^1(\R^n)$, one such quasi-continuous representative is 
\begin{equation} \label{eq:qcrep}
\tilde u(x) := \lim_{r \to 0^+} u_r(x), \quad \mbox{ where }  u_r(x) := \frac{1}{|B_r(x)|}\int_{B_r(x)} u(y)\,\rd y,
\end{equation}
 and the convergence in \eqref{eq:qcrep} holds q.e. Note also that, by \eqref{eq:aeqe}, if $\tilde u_1$ and $\tilde u_2$ are quasi-continuous representatives of $u$, then $\tilde u_1=\tilde u_2$ q.e.

Whenever $u\in L^2(\R^n)$, the convergence \eqref{eq:qcrep} holds a.e.\ (since a.e.\ $x\in \R^n$ is a Lebesgue point) and $\|u - u_r\|_{L^2(\R^n)}\to 0$ as $r\to 0^+$. This last fact follows by the dominated convergence theorem since $|u(x)-u_r(x)|^2$ is dominated by $(|\cM u(x)|+|u(x)|)^2$,  for a.e.\ $x\in \R^n$, where $\cM u$ is the maximal function
$$
\cM u(x) := \sup_{r>0}\frac{1}{|B_r(x)|}\int_{B_r(x)} |u(y)|\,\rd y,
$$
which satisfies (e.g., \cite[Thm.~1.1.1]{adams2012function}) $\|\cM u\|_{L^2(\R^n)}\leq C\|u\|_{L^2(\R^n)}$, for some $C>0$ independent of $u$.
Where $\chi_{B_r(0)}$ denotes the characteristic function of $B_r(0)$ and 
$
\chi_r:= \frac{\chi_{B_r(0)}}{|B_r(0)|},
$ 
for $r>0$, $u_r$ can be written as the convolution
$
u_r = \chi_r*u
$
and, if $u\in H^1(\R^n)$ then 
$
\nabla u_r = \chi_r * \nabla u.
$
From this we deduce by dominated convergence that also $u_r\to u$ in $H^1(\R^n)$ as $r\to 0^+$ if $u\in H^1(\R^n)$.  

\subsubsection{Subspaces of $H^s(\R^n)$}  \label{sec:subspaces}
We need also certain subspaces of $H^s(\R^n)$. For $s\in\mathbb{R}$ and a domain $\Omega\subset \R^n$, let  
 \begin{equation} \label{eq:tildeHs}
 \widetilde{H}^s(\Omega):={\overline{C^{\infty}_0(\Omega)}}^{H^s(\R^n)}.
 \end{equation}
 Importantly, if $u\in \widetilde{H}^1(\Omega)$, then $u=0$ a.e. on $\mathbb{R}^n\setminus\Omega$ and $u|_\Omega\in H_0^1(\Omega)$. Conversely, if $u\in H_0^1(\Omega)$ and the definition of $u$ is extended to $\R^n$ by setting $u=0$ on $\R^n\setminus \Omega$, then the extended $u\in \widetilde H^1(\Omega)$. 
 
 For any closed set $F\subset\mathbb{R}^n$, we set
\begin{equation} \label{eq:HsF}
H^s_F:=\{u\in H^s(\mathbb{R}^n): \text{supp}(u)\subset F\},
\end{equation}
noting that $H^s_F=\{0\}$, if $\mathrm{int}(F)$, the interior of $F$, is empty and $s$ is large enough (see \cite{HeMo:16}); in particular, $H^{-1}_F=\{0\}$ if and only if $\capp(F)=0$ \cite[Thm.~3.15]{HeMo:16}. Note that  
\begin{equation} \label{eq:cap}
\capp(F)>0 \quad \mbox{ if } \quad \dim_H(F)>n-2, 
\end{equation}
but not if $\dim_H(F)<n-2$ \cite[Thm.~2.12]{HeMo:16}. Further  (see, e.g., \cite[Proof of Prop.~2.5(vi)]{Paperdn}),
\begin{equation} \label{eq:boundH}
F \mbox{ compact and } \intt(F)\neq \emptyset \quad \Rightarrow \quad \dimH(\partial F) \geq n-1,
\end{equation}
so that $\capp(\partial F)>0$.

The subspace $(\widetilde H^{-s}(F^c))^\perp \subset H^{-s}(\R^n)$, where $\perp$ denotes orthogonal complement, is the realisation of the dual space $(H^s_F)^*$ via the duality pairing \eqref{eq:dual} restricted to $(\widetilde H^{-s}(F^c))^\perp \times H^s_F$ \cite[Thm.~3.15]{ChHeMo:17}, 
enabling us to identify $(H^s_F)^*=(\widetilde H^{-s}(F^c))^\perp$ and $\left((\widetilde H^{-s}(F^c))^\perp\right)^*=H^s_F$.

For all $s\in\mathbb{R}$ and every domain $\Omega\subset \R^n$ we have
\begin{equation}\label{4}
\widetilde{H}^s(\Omega)\subset H^s_{\overline{\Omega}},
\end{equation}
with equality in many cases, in particular if $\Omega$ is a $C^0$ domain \cite[Thm.~3.29]{mclean2000strongly}, \cite[Lem.~3.15]{ChHeMo:17}.

\bibliography{bibliography}
\bibliographystyle{plain}
\end{document}